\documentclass[11pt,leqno]{extarticle}
\usepackage[utf8]{inputenc}
\usepackage[margin=1in]{geometry}
\usepackage{enumitem,amssymb,amsmath,xcolor,cancel,hyperref,graphicx,etoolbox,setspace,subcaption,lmodern,yfonts,float,colortbl,dsfont,centernot,comment, amsthm}
\usepackage{tabularx}
\usepackage{rotating}
\usepackage{indentfirst}
\newtheorem{theorem}{Theorem}
\newtheorem{lemma}{Lemma}
\newtheorem{prop}{Proposition}
\newtheorem{remark}{Remark}[section]
\newcommand{\norm}[1]{\left\lVert#1\right\rVert}

\usepackage{authblk}

\allowdisplaybreaks
\counterwithin{equation}{section}
\counterwithin{theorem}{section}
\counterwithin{lemma}{section}
\counterwithin{prop}{section}
\title{A study of general reaction-advection-diffusion equations describing dynamics between target, partaker, and guardian.}
\author[1]{Madi Yerlanov\footnote{Corresponding author}} 
\author[2]{Nancy Rodr{\'\i}guez}
\affil[1]{University of Colorado Boulder, Department of Applied Mathematics,\authorcr 1111 Engineering Center, Boulder, CO, 80309, \authorcr email: \href{mailto:madi.yerlanov@colorado.edu}{madi.yerlanov@colorado.edu}}
\affil[2]{University of Colorado Boulder, Department of Applied Mathematics,\authorcr 1111 Engineering Center, Boulder, CO, 80309, \authorcr email: \href{mailto:rodrign@colorado.edu}{rodrign@colorado.edu}}
\date{\today}
\begin{document}
\maketitle
\begin{abstract}
This paper introduces a reaction–advection–diffusion system that models interactions among three actors: a target, a partaker, and a guardian. The framework is versatile, capturing phenomena ranging from the emergence and movement of crime hotspots in urban areas to shifts in public attitudes during critical events as individuals and control units move through space. We prove local and global existence of solutions under realistic assumptions and showcase the model through two applications: protest dynamics driven by new demonstrators joining and escalating hostility in enclosed environments such as classrooms or offices. Numerical simulations highlight the resulting complex spatial patterns and temporal dynamics. 
\end{abstract}
\textbf{Key Words}: reaction-advection-diffusion equations, sociological modeling, routine activity theory, well-posedness.\\
\textbf{2020 Mathematics Subject Classification}: 35Q91 (Primary), 91C99, 37M05 (Secondary).
\section{Introduction}

Routine activities theory posits that a crime is likely to happen when three elements come together: a suitable target, a motivated offender, and a lack of a guardian \cite{miro2014routine}. This shifts the focus from the criminal as a person to the situation in which the crime occurs, mainly considering factors like availability, proximity, and exposure \cite{enwiki:1299623116}. This idea can also apply to other areas of sociology, and here, we substitute ``offender" with a more neutral term ``partaker". For example, the spread of a rumor or conspiracy in a community depends on the gullibility of the audience (target), the persuasiveness of the person spreading it (partaker), and whether there is anyone to challenge it (guardian). In this way, we can think of the interaction more broadly as any situation in which one person benefits from another's change in exposure. In this work, we model the relationships between targets, partakers, and guardians through reaction–advection–diffusion systems, to understand how their interactions change over time and across space.

Reaction–diffusion systems have long been applied to model processes in biology, ecology, and chemistry, most commonly involving interactions between two agents, such as predator and prey or competing species. They have also generated great interest in the mathematical community as a subject of study in nonlinear dynamics, pattern formation, and bifurcation theories; as a result, extensive analysis tools exist. See examples of both theory and application here \cite{cantrell2004spatial,cosner2008reaction, lam2022introduction}. A natural extension of these models is the inclusion of advection terms, which represent directed movement, such as taxis or cross-diffusion. In the chemotaxis setting, cells move along the gradient of a chemical attractant: \textit{e.g.}, soil bacteria seek the concentration of nutrients \cite{gharasoo2014chemotactic}. Reaction-advection-diffusion (RAD) systems, therefore, allow for more complex and realistic dynamics. At the same time, these systems are more challenging to analyze because the advection term can counteract diffusion-induced smoothness and may lead to blow-up.  

Although RAD systems were first used in biology and physical sciences, they have also been applied to social settings. An important example and the main motivation for this paper is the urban crime model introduced in 2008 by a team at the University of California, Los Angeles \cite{short2008statistical}. The model describes the interactions between households and burglars. In addition to the use of routine activities theory \cite{felson2017routine}, the model is based on repeat victimization, which posits that susceptibility diffuses to neighbors \cite{anselin2000spatial,bowers2005domestic, johnson1997new, short2009measuring}, and the broken windows effect, which suggests that an offense increases attractiveness \cite{braga2015can, sampson2004seeing,WilsonKelling1982}. The system was later extended to include police agents as well \cite{jones2010statistical, pitcher2010adding}. We state the version of this model with law enforcement proposed in \cite{jones2010statistical}
\begin{equation}\label{eq:E}
    \left\{
    \begin{aligned}
        \frac{\partial A}{\partial t}=&D_A\Delta A - A+\alpha+A\rho, &&\text{ in } \Omega\times(0, T], \\
        \frac{\partial \rho}{\partial t}=&\nabla \cdot \left(D_\rho \nabla \rho-2 \rho\nabla \ln A\right)-A\rho+\beta - \rho u, &&\text{ in } \Omega\times(0, T], \\
        \frac{\partial u}{\partial t}=&\nabla \cdot \left(D_u \nabla u-\chi u\nabla \ln A\right), &&\text{ in } \Omega\times(0, T], \\
         \frac{\partial A}{\partial \textbf{n}}=& \frac{\partial \rho}{\partial \textbf{n}}= \frac{\partial u}{\partial \textbf{n}}=0, &&\text{ on }\partial\Omega, \\
         (A,\rho,&\,u)(\textbf{x}, 0)=(A_0, \rho_0,u_0)\geq, \not\equiv 0, && \text{ in }\Omega.
    \end{aligned}
    \right.
\end{equation}

System \eqref{eq:E} is a RAD system with three agents $(A,\rho,u)$, satisfying certain initial and boundary conditions on a bounded domain $\Omega\subset \mathbb{R}^n$ (with piecewise smooth boundary $\partial\Omega$) and a final time $T$, which can be infinite. The scalar fields represent the {\it attractiveness value}, $A$, which indicates the probability that a burglar will commit an offense at a particular location, the {\it density of crime agents}, $\rho$, who move randomly but also tend to seek regions with high attractiveness, and the {\it density of police agents}, $u$, who have similar movement mechanics as the burglars. Repeat victimization is manifested through the $\Delta A$ term, and the broken window effect is modeled by the $+A\rho$ term. The attractiveness value has a growth rate of $-A+\alpha$, whereas crime agent density has only a constant rate of $+\beta$. Law enforcement resources are assumed to be fixed, so the total mass of $u$ is preserved over time. Crime agents leave the region after committing a crime or being removed by police $-\rho(A+u)$. 

The system \eqref{eq:E} provides a framework that can be generalized beyond crime and applied to a wide range of sociological domains, which serves as the central motivation of this paper.

\subsection{A Target-Partaker-Guardian Model}
In this paper, we consider a general RAD framework that describes the target-partaker-guardian relationship.   The general model, which will be the focus of this paper, is given below

\begin{equation}\label{eq:GenModel}
    \left\{
    \begin{aligned}
        \frac{\partial u}{\partial t}&=D_u\Delta u + f_1(u,v,w)+g_1(u,v,w)+h_1(u), &&\text{ in } \Omega\times(0,T],\\
        \frac{\partial v}{\partial t}&=\nabla \cdot \big[D_v \nabla v-v\chi_1(u,v,w)\nabla u \big]+ f_2(u,v,w)+g_2(u,v,w)+h_2(v),&&\text{ in } \Omega\times(0,T],\\
        \frac{\partial w}{\partial t}&=\nabla \cdot \big[D_w \nabla w-w \chi_2(u,v,w) \nabla u \big],&&\text{ in } \Omega\times(0,T],\\
         \frac{\partial u}{\partial \mathbf{n}}&= \frac{\partial v}{\partial \mathbf{n}}= \frac{\partial w}{\partial \mathbf{n}}=0,&&\text{ on }\partial\Omega,\\
         u(\cdot&,0)=u_0\geq0,\,v(\cdot,0)=v_0\geq0,\,w(\cdot,0)=w_0\geq0,&&\text{ in }\Omega\times\{0\}.
    \end{aligned}
    \right.
\end{equation}

In system \eqref{eq:GenModel}, $u$ represents the target, $v$ the partaker, and $w$ the guardian. Note that $T$ and $\Omega \subset \mathbb{R}^n$ are the same as in system \eqref{eq:E}, although the natural choice for $n$ is 2, which is the dimension used in the proofs later; hence, it is fixed throughout the paper. The partaker seeks to increase the target’s exposure or intensity, while the guardian works to reduce the influence of the partakers. Unlike in standard predator-prey models, the number of guardians does not grow through interactions, nor are there natural birth or death rates. As a result, the total number of guardians remains constant.  The homogeneous Neumann boundary conditions are standard for these models, which shows that the interacting population is isolated. We comment on each of the terms and assumptions involved, which the predator-prey models strongly inspire. 

The {\it diffusion terms} $D_v \Delta v$ and $D_w \Delta w$ represent the random movement of partakers and guardians. The term $D_u \Delta u$ can be interpreted as ordinary diffusion or as the spread of attractiveness, susceptibility, or a similar abstract quantity. In original crime modeling, this diffusion captures repeat and near-repeat victimization: The risk of being targeted extends to nearby connections; for instance, if one household is highly attractive to perpetrators, its neighbors also become somewhat attractive. In this paper, we assume that the diffusion rates are positive constants; however, the model can be generalized to consider a heterogeneous spread. 

The {\it advection terms} $\chi_1$ and $\chi_2$ represent the rates at which nontarget agents move in response to the target gradient. They describe how sensitive partakers and guardians are to the presence of targets.  We drop the dependence on $v$ and $w$ in the definition. The {\it growth and decay terms} $h_1$ and $h_2$ represent the natural growth or death of targets and partakers. The guardians, $ w$, have no growth or decay terms, only movement, reflecting the assumption that their total number is capped. 

The {\it victimization terms} $f_1$ and $f_2$ describe how interactions between targets and partakers change their states. In the crime setting, $f_1$ is positive (the target becomes more susceptible after being victimized), while $\text{sign}(f_2) = -\text{sign}(f_1)$ (for example, partakers may be removed from the system after committing their acts). This mirrors the infection mechanism in SIR models, where susceptible individuals become infected. We also assume that the reaction is scaled to show how a partaker's action relatively affects the ``attractiveness" value by a constant, $\gamma$. Thus, we redefined functions in the following way: $f_1(u,v,w)=\gamma f(u,v,w)$ and $f_2(u,v,w)=-f(u,v,w)$ for some non-negative function $f$.

The functions $g_1$ and $g_2$, the {\it control terms}, capture the effect of the guardians on the system. Typically, $g_2 \leq 0$ since partakers are removed when they encounter guardians. The function $g_1$ counteracts victimization, often having the opposite sign of $f_1$ (or $f$ in our case). In some settings, $g_1$ may be ignored, for example, when police cannot directly reduce a household’s attractiveness after a crime.  Finally, the reaction terms vanish when their corresponding population is zero, e.g., $g_1(0,v,w)=0$. This assumption directly shows the involvement of scalar fields in the corresponding reactions. If the scalar field is zero, then no interaction is possible. To enforce this, each function is multiplied by its corresponding variable. For example, as $f$ appears in both $u$ and $v$ equations, we have $f(u,v,w) \mapsto uvf(u,v,w)$. We also extend this assumption to $h_2$ to emphasize the presence of the logistic growth, \textit{i.e.}, $h_2(v)\mapsto vh_2(v)$.

With these assumptions in mind, we consider the following model
\begin{equation}\label{eq:themainsys}
    \left\{
    \begin{aligned}
        \frac{\partial u}{\partial t}&=D_u\Delta u +u \big[\gamma vf(u,v,w)- g_1(u,v,w)\big]+h_1(u), &&\text{ in } \Omega\times(0,T],\\
        \frac{\partial v}{\partial t}&=\nabla \cdot \big[D_v \nabla v-v\chi_1(u)\nabla u \big]\\
        &+ v\big[-uf(u,v,w)-g_2(u,v,w)+h_2(v)\big],&&\text{ in } \Omega\times(0,T],\\
        \frac{\partial w}{\partial t}&=\nabla \cdot \big[D_w \nabla w-w \chi_2(u) \nabla u \big],&&\text{ in } \Omega\times(0,T],\\
                 \frac{\partial u}{\partial \mathbf{n}}&= \frac{\partial v}{\partial \mathbf{n}}= \frac{\partial w}{\partial \mathbf{n}}=0,&&\text{ on }\partial\Omega,\\
         u(\cdot&,0)=u_0\geq0,\,v(\cdot,0)=v_0\geq0,\,w(\cdot,0)=w_0\geq0,&&\text{ in }\Omega\times\{0\}.
    \end{aligned}
    \right.
\end{equation}
  
 {\it Outline:} This paper is divided into two main parts: theory and application. In section \ref{Sec:MainRes}, we prove the local and global existence of system \eqref{eq:themainsys} under certain assumptions. The proof follows the steps developed for three-component predator-prey systems with taxis \cite{qiu2023dynamics, wang2015stationary, wu2018dynamics} and relies on the general RAD theory \cite{tao2015boundedness, winkler2010absence, winkler2010aggregation}. The main difference lies in the inclusion of a guardian, whose interactions differ from those of a typical predator, as described above. Then, in section \ref{Sec:linst}, we provide a basic linear stability analysis showing conditions under which both $u$ reaches the baseline constant value and $v$ effectively becomes extinct. In section \ref{Sec:App}, we demonstrate the generality of system \eqref{eq:themainsys} for possible applications. We provide two distinct examples that describe critical social situations at the local population level. The scenarios include the spread of protest propensity and the growth of school bullying.

\section{Local and Global Existence}\label{Sec:MainRes}
In this section, we state and prove the local and global existence of system \eqref{eq:themainsys} under general assumptions. The general idea is first to demonstrate the local existence and then to find a uniform bound that does not depend on time. We focus on two-dimensional domains, which are the physically relevant geometries. The additional (and perhaps overlapping with model assumptions) mathematical hypotheses required for the proof are as follows.

\begin{enumerate}
    \item[(H1)] All parameters ($D_u,\,D_v,\,D_w$, and $\gamma$) and following bounds ($A_1,\,B_1,\,A_2,\,B_2,\,E_1,\,E_2$, and $F$) are positive.
    \item[(H2)]  \textit{The advection terms grow at most linearly in $v$ and $w$, respectively}, \textit{i.e.}, $\chi_i$ is bounded,  $\chi_i(s)\leq E_i$ for $i=1,\,2.$
    \item[(H3)] The \textit{intrinsic growth terms are saturated}, \textit{i.e.}, the function $h_i$ is bounded, $h_i(s)\leq A_i  - B_i s$ for $i=1,\,2$. 
    \item[(H4)] $g_i$ is continuously differentiable on $[0,\infty)^3$ for $i=1,\,2$; $f$ is continuously differentiable on $[0,\infty)^3$, and $h_i$ and $\chi_i$ are continuously differentiable on $[0,\,\infty)$ for $i=1,2$. 
    \item[(H5)] \textit{The guardianship/victimization effects are monotone}, \textit{i.e.}, $g_i$ for $i=1,2$, and $f$ are non-negative. 
    \item[(H6)] The \textit{victimization effect is saturated and/or controlled by a guardian:} $\gamma vf(u,v,w)\leq F+g_1(u,\,v,\,w)$ for all $u,\,v,\,w\geq 0$. 
\end{enumerate}

\begin{remark}
Regarding (H6), one may ask how a guardian can control victimization. For example, let $f(u,v,w) = 1$, giving the reaction term that appears in the $A,\rho$ equations in system \eqref{eq:E}, \textit{i.e.}, $\pm uv$, and
$$
g_1(u,v,w) = \frac{J v (1+w)}{1 + e^{-K(v-L)}},
$$
with $K \gg 1$, $J > \gamma$, and $L > 0$. If the number of partakers becomes too large, guardians immediately reduce attractiveness in proportion to all relevant fields. This links $g_1$ and $g_2$ in \eqref{eq:GenModel} so that removing partakers and reducing attractiveness occur simultaneously.
\end{remark}

With these assumptions in place, we are now ready to state our first main result.

\begin{theorem}[Local existence]\label{thm:locex}
    Assume that hypotheses (H1)-(H6) hold and the initial data satisfies $u_0,\,v_0,\,w_0\in [W^{1,p}(\Omega)]^3$ for some $p>n$, where $n=2$ is the dimension of the domain. 
    \begin{enumerate}
        \item There exists a $T_{\max}>0$ such that system \eqref{eq:themainsys} has a unique non-negative classical solutions $(u,v,w)\in[C(\overline{\Omega}\times[0,\,T_{\max}))\cap C^{2,1}(\overline{\Omega}\times(0,T_{\max}))]^3$.
    \item The total mass is bounded:
    \begin{equation*}
        \int_\Omega u(x,t)dx\leq C_1,\quad \int_\Omega v(x,t)dx\leq C_2,\quad \int_\Omega w(x,t)dx= C_3,\quad t\in(0,T_{\max}),
    \end{equation*}
    where $$C_1=\max\left\{\int_\Omega u_0+\gamma v_0dx,\,\frac{2\gamma B_1C_2+|\Omega|(A_1+A_2)}{B_1}\right\},C_2=\max\left\{\int_\Omega v_0dx,\,\frac{A_2|\Omega|}{B_2}\right\},$$
    and 
    $$C_3=\int_\Omega w_0(x)dx.$$
    \item There exists a positive constant $K$ such that
    \begin{equation*}
        0\leq u(x,t)< K,\quad v\geq0,\quad w\geq0,\quad x\in\overline{\Omega},\quad t\in[0,T_{\max}).
    \end{equation*}
    \item     
    If $T_{\max}<\infty$, then 
    \begin{equation*}
        \norm{u(\cdot,t)}_{W^{1,\infty}(\Omega)}+\norm{v(\cdot,t)}_{L^\infty(\Omega)}+\norm{w(\cdot,t)}_{L^\infty(\Omega)}\to \infty\quad\text{ as } t\nearrow T_{\max}.
    \end{equation*}
    \end{enumerate}

\end{theorem}
\begin{proof}(\textbf{Theorem \ref{thm:locex}}) First note that the $W^{1,p}$ assumption on the initial data guarantees that $u_0$, $v_0$, and $w_0$ are H\"older continuous by the Sobolev embedding theorem, hence we have regularity in the classical sense. Let $\vec{z}:=(u,v,w)$, then system \eqref{eq:themainsys} can be rewritten as follows

\begin{equation*}
    \left\{
    \begin{aligned}
        \frac{\partial \vec{z}}{\partial t}&=\nabla \cdot(Q\nabla\vec{z})+R(\vec{z})&&\text{ in } \Omega\times(0,T],\\
        \frac{\partial \vec{z}}{\partial \mathbf{n}}&=0&&\text{ on } \partial\Omega,\\
        \vec{z}(\cdot,&\, t)=(u_0\geq0,v_0\geq0,w_0\geq0)&&\text{ in } \Omega\times\{0\},
    \end{aligned}
    \right.
\end{equation*}
where
\begin{equation*}
    Q=\begin{bmatrix}
        D_u&0&0\\[0.5em]
        -v\chi_1(u)&D_v&0\\[0.5em]
        -w\chi_2(u)&0&D_w
    \end{bmatrix},
    \quad
    R=\begin{bmatrix}
     \gamma uvf(u,v,w)-ug_1(u,v,w)+h_1(u)\\[0.5em]
     - uvf(u,v,w)-vg_2(u,v,w)+vh_2(v)\\[0.5em]
     0
    \end{bmatrix}.
\end{equation*}
Note that $Q$ is lower triangular with positive diagonal entries by (H1), hence the system is normal parabolic. Then, parts \textit{1} and \textit{4} of Theorem \ref{thm:locex} follow from Amman's Theorem \cite{amann_dynamic_1989}, \cite{amann_dynamic_1990}, \cite{amann_nonhomogeneous_1993}, rephrased in \cite{wu2016global}. 

(H1) and (H4) allow us to apply the maximum principle. We obtain that $0$ is a lower solution. Then, the non-negativity of the solution follows from the non-negativity of the initial conditions. This proves part \textit{3}.

To prove part {\it 2}, note that $w$ has a constant mass by an application of the divergence theorem. Now, let $U=U(t):=\int_\Omega u(x,t)dx$ and $V=V(t):=\int_\Omega v(x,t)dx$. Using boundary conditions and (H3), we have

\begin{equation*}
    \begin{aligned}
        \frac{dV}{d t}&\leq \int_\Omega vh_2(v)dx\leq A_2\int_\Omega v dx-B_2\int_\Omega v^2dx\\
        &\leq A_2\int_\Omega vdx -\frac{B_2}{|\Omega|}\bigg(\int_\Omega vdx\bigg)^2\leq A_2 V-\frac{B_2}{|\Omega|}V^2.
    \end{aligned}
\end{equation*}

Note that the (H5) assumption allows us to remove the negative reaction terms. Cauchy-Schwartz allows us to bound $L^2(\Omega)$ with a $L^1(\Omega)$ estimate. After the solving the differential inequality we get that $V\leq C_2:=\max\left\{\int_\Omega v_0dx,\frac{A_2|\Omega|}{B_2}\right\}$.  

To get the inequality for $U$, we go through the same steps as before, but now we focus on $(U+\gamma V)$.  We obtain that
\begin{equation*}
    \begin{split}
        (U+\gamma V)_t&\leq\int_\Omega h_1(u)dx+\gamma \int_\Omega vh_2(v)dx\leq A_1|\Omega|-B_1 \int_\Omega u dx +  \gamma A_2\int_\Omega v dx-\gamma B_2\int_\Omega v^2dx\\
        &\leq A_1 |\Omega|-B_1U+\gamma A_2V-\gamma \frac{B_2}{|\Omega|}V^2\\
        &=-B_1(U+\gamma V) +\Big(A_1|\Omega|+\gamma V(A_2+B_1)\Big)\\
        &=-B_1(U+\gamma V) +\Big(A_1|\Omega|+\gamma C_2(A_2+B_1)\Big).
    \end{split}
\end{equation*}
We solve the differential equation and obtain that 
$$U\leq(U+\gamma V)\leq C_1:=\max\left\{\int_\Omega u_0+\gamma v_0dx,\frac{A_1|\Omega|+\gamma C_2(A_2+B_1)}{B_1}\right\}.$$
This proves part \textit{2} of Theorem \ref{thm:locex}.  Finally, we treat the first equation as a scalar with respect to $u$. Note that by (H3), if $u>A_1/B_1$, then $h_1<0$, and we have
\begin{equation*}
    \begin{aligned}
         \frac{\partial u}{\partial t}&=D_u\Delta u+u\Big(\gamma v f(u,v,w)-g_1(u,v,w)\Big)+h_1(u)\\
      &  \leq \Delta u+ Fu,
    \end{aligned}
\end{equation*}
where $F$ is a bound from (H6). This, in combination with Theorem 3.1 in Alikakos \cite{alikakos1979lp}, implies that there exists $C$ such that 
\begin{equation*}
    \sup_{t\in[0,T_{\max})}\norm{u}_\infty< C\max\left\{1,\frac{A_1}{B_1},\sup_{t\in[0,T_{\max})}\norm{u}_1,\norm{u_0}_\infty\right\}\leq K:= C\max\left\{1,\frac{A_1}{B_1},C_1,\norm{u_0}_\infty\right\}.
\end{equation*}
 Note that $C_1$ does not depend on $t$, hence the bound is uniform. This concludes part \textit{3} of Theorem \ref{thm:locex}.
\end{proof}

\begin{remark}
    In proving the $L_1$-bounds of Theorem \ref{thm:locex}, if one has different bounds on the growth rate, for example, quadratic on $u$, then one needs to complete a square (by adding and subtracting terms) to obtain an expression of the form
    \begin{equation*}
        (U+\gamma V)_t\leq -K_1(U+\gamma V)^2+K_2(U+\gamma V)
    \end{equation*}
    for some constant $K_1$ and $K_2$ independent of $U$ and $V$. This can be extended to higher powers; however, the expressions for constants become increasingly algebraically complex.
\end{remark}

The local existence, in combination with some a priori estimates, will give our second main result. 
\begin{theorem}[Global existence]\label{thm:globex}
    Assume that the hypotheses of Theorem \ref{thm:locex} hold.
 Then, system \eqref{eq:themainsys} has a unique global non-negative classical solutions $(u,v,w)\in[C(\overline{\Omega}\times[0,\infty))\cap C^{2,1}(\overline{\Omega}\times(0,\infty))]^3$.
\end{theorem}

To prove Theorem \ref{thm:globex}, we use part \textit{4} of Theorem \ref{thm:locex}, which indicates that it suffices to establish uniform bounds independent of $T_{\max}$. We begin with $u$, since it does not contain an advection term. Moreover, because $\nabla u$ appears in the other equations, we must control not only $u$ itself but also its gradient. For this reason, we work with bounds in the $W^{1,\infty}$-norm.

\begin{lemma}\label{lem:W1infonu}
     Assume the hypotheses and constants of Theorem \ref{thm:locex}.  Then, for any
$\tau\in(0,\min\{1,T_{\max}\})$, there exists a constant $M_1(\tau)$ such that
      \begin{equation*}
          \norm{u(\cdot,t)}_{W^{1,\infty}(\Omega)}\leq M_1(\tau)\quad \forall t\in(\tau,T_{\max}).
      \end{equation*}
\end{lemma}

\begin{proof}
Let $\theta\in\left(\frac{1}{2}(1+\frac{n}{p}),1\right)$. Rewrite the first equation as
\begin{equation*}
    u_t=D_u\Delta u-u +\phi(u,v,w)+h_1(u),\quad \phi(u,v,w):=\gamma u v f(u,v,w)-ug_1(u,v,w)+u.
\end{equation*}
 We use the variation of constants formula using the sectorial operator $\Lambda=-\Delta$:
\begin{equation*}
\begin{split}
        u(\cdot,t)&=e^{-(D_u\Lambda+1)t}u_0+\int_0^te^{-(t-s)( D_u\Lambda+1)}\phi(u(\cdot,s),v(\cdot,s),w(\cdot,s))ds\\
        &\quad +\int_0^te^{-(t-s)( D_u\Lambda+1)}h_1(u(\cdot,s))ds.
\end{split}
\end{equation*} 

 At each line below, when we move from one inequality to another, constants in front of the expressions may vary; however, we still denote them using $c$. When multiple terms are present, we use $c$ to denote the maximum among the constants. The explicit dependence on $\tau$ is noted. Using Lemma \ref{thm:semigroup}, we get that there exist $c$ and $\nu$ such that

\begin{equation*}
    \begin{aligned}
        \norm{u(\cdot,t)}_{W^{1,\infty}(\Omega)}
        &\leq c\norm{(D_u\Lambda+1)^\theta u}_{L^p(\Omega)}
        \\
        &\leq c\Big[t^{-\theta}e^{-\nu t}\norm{u_0}_{L^p(\Omega)}
        \\
        &\quad+\int_0^t(t-s)^{-\theta}e^{-\nu(t-s)}\norm{\phi(u(\cdot,s),v(\cdot,s),w(\cdot,s))}_{L^p(\Omega)}ds\\
        &\quad +\int_0^t(t-s)^{-\theta}e^{-\nu(t-s)}\norm{h_1(u(\cdot,s))}_{L^p(\Omega)}ds\Big].
    \end{aligned}
\end{equation*}
Next, we use Lipschitz continuity of $\phi(u,v,w)$ and $h(u)$, which is a consequence of (H4). Note that $\phi(0,v,w)=0$, and 

\begin{equation*}
    \begin{aligned}
     \norm{u(\cdot,t)}_{W^{1,\infty}(\Omega)}
        &\leq c\Big[t^{-\theta}e^{-\nu t}\norm{u_0}_{L^p(\Omega)}\\
        &\quad+\int_0^t(t-s)^{-\theta}e^{-\nu(t-s)}\norm{u(\cdot,s)}_{L^p(\Omega)}ds\\
 &\quad+\int_0^t(t-s)^{-\theta}e^{-\nu(t-s)}ds\Big].
            \end{aligned}
\end{equation*}

We employ part \textit{3} of Theorem \ref{thm:locex} to have the following bound $$\norm{u(\cdot,s)}_{L^p(\Omega)}\leq |\Omega|^{1/p}\norm{u(\cdot,s)}_{L^{\infty}(\Omega)}\leq c.$$ The bound is independent of the time. We also use $W^{1,p}(\Omega)$-boundedness of $u_0$. We have a substitution $\sigma:=t-s$ and extend the range of the second interval to $\infty$ to apply the definition of the $\Gamma$-function. 

    \begin{equation*}
    \begin{aligned}
     \norm{u(\cdot,t)}_{W^{1,\infty}(\Omega)}    
        &\leq c\left[t^{-\theta}+\int_0^t(t-s)^{-\theta}e^{-\nu(t-s)}ds\right]\leq c\left[t^{-\theta}+\int_0^\infty\sigma^{-\theta}e^{-\nu\sigma}d\sigma\right]\\
        &\leq c\left[t^{-\theta}+\frac{\Gamma(1-\theta)}{\nu^{1-\theta}}\right]\leq c(\tau^{-\theta}+1)=:M_1(\tau)\quad \forall t\in(\tau,T_{\max}).
    \end{aligned}
\end{equation*}

Note that $M_1$ depends only on $\tau,\Omega, p, D_u,\norm{u_0}_{W^{1,p}}$ and the parameters defined in $(H1)-(H6)$. 
\end{proof} 

The bounds for $v$ and $w$ are more difficult to establish because their equations include advection terms. Our strategy is to first derive bounds in weaker norms and then upgrade these estimates to obtain uniform bounds.

\begin{lemma}\label{lem:L2onv}
Assume the hypotheses and constants of Lemma \ref{lem:W1infonu}. Then, for any
$\tau\in(0,\min\{1,T_{\max}\})$, there exists a constant $N_1(\tau)$ such that
      \begin{equation*}
          \norm{v(\cdot,t)}_{L^{2}(\Omega)}\leq N_1(\tau)\quad \forall t\in(\tau,T_{\max}).
      \end{equation*}
\end{lemma}

\begin{proof}
We differentiate the $L^2$-norm of $v$ with respect to $t$ and apply integration by parts, which is a consequence of Green's first identity 
\begin{equation*}
    \begin{aligned}
        \frac{d}{dt}\int_\Omega v^2(x,t) dx
        &=2\int_\Omega v(x,t)\frac{\partial v(x,t)}{\partial t}dx\\
        &\leq 2D_v\int_\Omega v\Delta vdx-2\int_\Omega v \nabla \cdot [v\chi_1(u)\nabla u]dx+2\int_\Omega v^2h_2(v)dx\\
        &=-2D_v\int_\Omega |\nabla v|^2dx+2\int_\Omega v\chi_1(u)\nabla v\cdot\nabla udx+2\int_\Omega v^2 h_2(v)dx.
                \end{aligned}
\end{equation*}
We work on the second integral separately. We apply Young's inequality for inner products
\begin{equation*}
    \begin{aligned}
        2\int_\Omega v\chi_1(u)\nabla v\cdot\nabla udx\leq
        D_v\int_\Omega |\nabla v|^2dx + \frac{1}{D_v}\int_\Omega \left(v\chi_1(u)|\nabla u|\right)^2dx.
    \end{aligned}
\end{equation*}
We put the right-hand side back into the original inequality, which cancels the $2$ in front of the first integral. The bounds for $h_2(v)$, $\chi_1(u)$ and $|\nabla u|$ are provided by (H3), (H2) and Lemma \ref{lem:W1infonu}, respectively. We group all the constants in front of $\int_\Omega v^2dx$ and label them $c_1(\tau)$
\begin{equation}\label{eq:v2norm1}
    \begin{aligned}
         \frac{d}{dt}\norm{v}_{L^2(\Omega)}^2\leq-D_v\norm{\nabla v}^2_{L^2(\Omega)}+c_1(\tau)\norm{v}^2_{L^2(\Omega)}.
        \end{aligned}
\end{equation}

We have to deal with the first term on the right-hand side. We apply Lemma \ref{thm:GNineq} and the general inequality $(a+b)^2\leq4(a^2+b^2)$. As before $c_2(\tau)$ may differ from line to line, but it is still proportional to $M_1(\tau)$ from Lemma \ref{lem:W1infonu}
\begin{equation}\label{eq:v2norm2}
    \begin{split}
(c_1(\tau)+1)\norm{v}_{L^2(\Omega)}^2&\leq c_2(\tau)\Big(\norm{
            \nabla v}_{L^2(\Omega)}^{1/2} \norm{v}_{L^1(\Omega)}^{1/2}+\norm{v}_{L^{1}(\Omega)}\Big)^2\\
            &\leq c_2(\tau)\Big( \norm{
            \nabla v}_{L^2(\Omega)} \norm{v}_{L^1(\Omega)}+\norm{v}_{L^{1}(\Omega)}^2\Big)\\
            &\leq D_v\norm{
            \nabla v}^2_{L^2(\Omega)} +c_2(\tau)\norm{v}_{L^{1}(\Omega)}^2,
    \end{split}
\end{equation}
where we applied Young's inequality on $ c_2(\tau)\norm{
            \nabla v}_{L^2(\Omega)} \norm{v}_{L^1(\Omega)}$.  
We add \eqref{eq:v2norm1} and \eqref{eq:v2norm2}, we have the following differential inequality
\begin{equation*}
    \begin{split}
        \frac{d}{dt}\norm{v}_{L^2(\Omega)}^2+\norm{v}_{L^2(\Omega)}^2\leq c_2(\tau).
    \end{split}
\end{equation*}

By the comparison principle, we get that 
\begin{equation*}
    \norm{v}_{L^2(\Omega)}^2\leq \max\Big\{\norm{v_0}^2_{L^2(\Omega)},c_2(\tau)\Big\}=:N_1^2(\tau).
\end{equation*}
Taking the square root on both sides completes the proof.
\end{proof}

\begin{lemma}\label{lem:L2onw}
Assume the hypotheses and constants of Lemma \ref{lem:W1infonu}. Then, for any
$\tau\in(0,\min\{1,T_{\max}\})$, there exists a constant $N_2(\tau)$ such that
      \begin{equation*}
          \norm{w(\cdot,t)}_{L^{2}(\Omega)}\leq N_2(\tau)\quad \forall t\in(\tau,T_{\max}).
      \end{equation*}
\end{lemma}

\begin{proof}
The proof is identical to the proof for Lemma \ref{lem:L2onv}.
\end{proof}

We prove the last weak bound.

\begin{lemma}\label{lem:L4onv}
From Lemma \ref{lem:L2onv}, let assumptions hold and constants be defined. Then, for any $\tau\in(0,\min\{1,T_{\max}\})$, there exists constant $N_3(\tau)$ such that
      \begin{equation*}
          \norm{v(\cdot,t)}_{L^{4}(\Omega)}\leq N_3(\tau)\quad \forall t\in(\tau,T_{\max}).
      \end{equation*}
\end{lemma}

\begin{proof}
 
The approach is very similar to the one in Lemma \ref{lem:L2onv}, hence the repeating details are omitted. $c_1(\tau)$ and $c_2(\tau)$ serve the same purpose as before.
\begin{equation*}
    \begin{split}
        \frac{d}{dt}\int_\Omega v^4(x,t) dx&=4\int_\Omega v^3(x,t)\frac{\partial v(x,t)}{\partial t}dx\\
        &\leq 4D_v\int_\Omega v^3\Delta vdx-4\int_\Omega v^3 \nabla \cdot [v\chi_1(u)\nabla u]dx +4\int_\Omega v^4h_2(v)dx\\
        &=-12D_v\int_\Omega v^2 |\nabla v|^2dx+12\int_\Omega v^3\chi_1(u)\nabla v\cdot\nabla udx +4\int_\Omega v^4 h_2(v)dx\\
        &=-3D_v\int_\Omega|\nabla v^2|^2dx+6\int_\Omega v^2\chi_1(u)\nabla v^2\cdot\nabla udx +4\int_\Omega v^4 h_2(v)dx\\
        &\leq -D_v\int_\Omega |\nabla v^2|^2dx+\int_\Omega \frac{9\chi_1(u)^2}{2D_v}v^4|\nabla u|^2dx +4\int_\Omega v^4 h_2(v)dx\\
        &\leq -D_v\int_\Omega |\nabla v^2|^2dx+c_1(\tau)\int_\Omega v^4dx\\
        &=-D_v\int_\Omega |\nabla v^2|^2dx+c_1(\tau)\int_\Omega v^4dx.
        \end{split}
\end{equation*}

We now apply Lemma \ref{thm:GNineq} to $v^2$
\begin{equation*}
    \begin{split}
 (c_1(\tau)+1)\norm{v^2}_{L^2(\Omega)}^2&\leq c_2(\tau)\bigg(\norm{
            \nabla v^2}_{L^2(\Omega)}^{1/2} \norm{v^2}_{L^1(\Omega)}^{1/2}+\norm{v^2}_{L^{1}(\Omega)}\bigg)^2\\
            &\leq c_2(\tau)\bigg( \norm{
            \nabla v^2}_{L^2(\Omega)} \norm{v^2}_{L^1(\Omega)}+\norm{v^2}_{L^{1}(\Omega)}^2\bigg)\\
            &\leq D_v\norm{
            \nabla v^2}^2_{L^2(\Omega)} +c_2(\tau)\norm{v^2}_{L^{1}(\Omega)}^2.
    \end{split}
\end{equation*}
Note that $\norm{v^2}_{L^2(\Omega)}^2=\norm{v}_{L^4(\Omega)}^4$ and $\norm{v^2}_{L^1(\Omega)}^2=\norm{v}_{L^2(\Omega)}^4$.
Combining everything, we have
\begin{equation*}
    \begin{split}
        \frac{d}{dt}\norm{v}_{L^4(\Omega)}^4+\norm{v}_{L^4(\Omega)}^4\leq c_2(\tau).
    \end{split}
\end{equation*}

By the comparison principle, we get that 
\begin{equation*}
    \norm{v}_{L^4(\Omega)}^4\leq \max\Big\{\norm{v_0}^4_{L^4(\Omega)},c_2(\tau)\Big\}=:N_3^4(\tau).
\end{equation*}
By taking the fourth root on both sides, we complete the proof.
   
\end{proof}

\begin{lemma}\label{lem:L4onw}
From Lemma \ref{lem:L2onw}, let assumptions hold and constants be defined. Then for any
$\tau\in(0,\min\{1,T_{\max}\})$, there exists a constant $N_4(\tau)$ such that
      \begin{equation*}
          \norm{w(\cdot,t)}_{L^{2}(\Omega)}\leq N_4(\tau)\quad \forall t\in(\tau,T_{\max}).
      \end{equation*}
\end{lemma}

\begin{proof}

The proof is identical to the proof for Lemma \ref{lem:L4onv}.
    
\end{proof}

We now have sufficient regularity to extrapolate to infinity. 

\begin{lemma}\label{lem:Linfonv}
From Lemmas \ref{lem:L4onv} and \ref{lem:L4onw}, let assumptions hold and constants be defined. Then for any $\tau\in(0,\min\{1,T_{\max}\})$, there exists constant $M_2(\tau)$ such that
      \begin{equation*}
          \norm{v(\cdot,t)}_{L^{\infty}(\Omega)}\leq M_2(\tau)\quad \forall t\in(\tau,T_{\max}).
      \end{equation*}
\end{lemma}

\begin{proof}

Let $\theta\in\left(\frac{1}{2}(1+\frac{n}{p}),1\right)$. Rewrite the second equation as
\begin{equation*}
    \begin{split}
            &v_t=D_v\Delta v-v-\nabla\cdot(v\chi_1(u)\nabla u) +\varphi(u,v,w),\\
            &\varphi(u,v,w):=- u v f(u,v,w)-vg_2(u,v,w)+vh_2(v)+v.
    \end{split}
\end{equation*}
We use the variation of constants formula using the sectorial operator $\Lambda=-\Delta$
\begin{equation*}
\begin{split}
        v(\cdot,t)&=e^{-(D_v\Lambda+1)t}v_0-\int_0^te^{-(t-s)( D_v\Lambda+1)}\nabla\cdot(v(\cdot,s)\chi_1(u(\cdot,s))\nabla u(\cdot,s))ds\\
        &+\int_0^te^{-(t-s)( D_v\Lambda+1)}\varphi(u(\cdot,s),v(\cdot,s),w(\cdot,s))ds\\
        &=:V_1+V_2+V_3.
\end{split}
\end{equation*}

We use semigroup results with $m=0$, $q=\infty$, and $p=4>n$. We need to prove that each of $V_i$ is bounded with the bounds dependent on $\tau$, but not $T_{\max}$. As usual, $c_i$ and $C_i(\tau)$ may vary between consecutive inequalities for $i=1,\,2,\,3$. All $C_i(\tau)$ are inversely proportional to $\tau$.

By Lemma \ref{thm:semigroup}, for $V_1$, let $\theta_1\in(\frac{n}{2p},1)$, then there exist $c_1>0$  and $\nu_1>0$ such that
\begin{equation*}
    \begin{split}
        \norm{V_1}_{L^{\infty}(\Omega)}&\leq c_1\norm{(D_v\Lambda+1)^{\theta_1} e^{-(D_v\Lambda+1)t}v_0}_{L^{p}(\Omega)}\leq c_1 t^{-\theta_1}e^{-\nu_1 t} \norm{v_0}_{L^p(\Omega)}\\
        &\leq C_1(\tau)\norm{v_0}_{W^{1,p}(\Omega)}\leq C_1(\tau).
    \end{split}
\end{equation*}

By Lemma \ref{thm:semigroup}, for $V_2$, let $\theta_2\in\left(\frac{n}{2p},\frac{1}{2}\right)$ and $\varepsilon\in\left(0,\frac{1}{2}-\theta_2\right)$, then there exist $c_2>0$ and $\mu>0$ such that

\begin{equation*}
    \begin{split}
        \norm{V_2}_{L^{\infty}(\Omega)}&\leq c\norm{(D_v\Lambda+1)^{\theta_2} V_2}_{L^{p}(\Omega)}\\
        &=c_2\norm{(D_v\Lambda+1)^{\theta_2}\int_0^te^{-(t-s)( D_v\Lambda+1)}\nabla\cdot\Big(v(\cdot,s)\chi_1(u(\cdot,s))\nabla u(\cdot,s)\Big)ds}_{L^{p}(\Omega)}\\
        &\leq c_2\int_0^t\norm{(D_v\Lambda+1)^{\theta_2}e^{-(t-s)( D_v\Lambda+1)}\nabla\cdot\Big(v(\cdot,s)\chi_1(u(\cdot,s))\nabla u(\cdot,s)\Big)}_{L^{p}(\Omega)}ds\\
        &=c_2\int_0^te^{-(t-s)}\norm{(D_v\Lambda+1)^{\theta_2}e^{-(t-s) D_v\Lambda}\nabla\cdot\Big(v(\cdot,s)\chi_1(u(\cdot,s))\nabla u(\cdot,s)\Big)}_{L^{p}(\Omega)}ds\\
        &\leq c_2\int_0^t(t-s)^{-\theta_2-\frac{1}{2}-\varepsilon}e^{-(t-s)(\mu+1)}\norm{v(\cdot,s)\chi_1(u(\cdot,s))\nabla u(\cdot,s)}_{L^{p}(\Omega)}ds\\
        &\leq C_2(\tau)\int_0^t(t-s)^{-\theta_2-\frac{1}{2}-\varepsilon}e^{-(t-s)(\mu+1)}ds\leq C_2(\tau)\int_0^\infty\sigma^{-\theta_2-\frac{1}{2}-\varepsilon}e^{-\sigma(\mu+1)}d\sigma\\
        &\leq C_2(\tau) \frac{\Gamma (\frac{1}{2}-\theta_2-\varepsilon)}{(\mu+1)^{(\frac{1}{2}-\theta_2-\varepsilon)}}\leq C_2(\tau),
    \end{split}
\end{equation*}
where we $C_2(\tau)$ is proportional to $M_1(\tau)$ and $N_3(\tau)$ from Lemmas \ref{lem:W1infonu} and \ref{lem:L4onv}, respectively.

For $V_3$, we know that $\varphi(u,v,w)$ is Lipschitz continuous. Let $\theta_3\in\left(\frac{n}{2p},1\right)$, then there exist $c_3>0$ and $\nu_2>0$ such that: 
\begin{equation*}
    \begin{split}
        \norm{V_3}_{L^{\infty}(\Omega)}&\leq c_3\norm{(D_v\Lambda+1)^{\theta_3} v}_{L^p(\Omega)}\\
        &\leq c_3\int_0^t(t-s)^{-\theta_3}e^{-\nu_2(t-s)}\norm{\varphi(u(\cdot,s),v(\cdot,s),w(\cdot,s))}_{L^p(\Omega)}ds\\
        &\leq c_3\int_0^t(t-s)^{-\theta_3}e^{-\nu_2(t-s)}\Big(\norm{u(\cdot,s)}_{L^p(\Omega)}+\norm{v(\cdot,s)}_{L^p(\Omega)}+\norm{w(\cdot,s)}_{L^p(\Omega)}\Big)ds.\\
            \end{split}
\end{equation*}
Lemmas \ref{lem:W1infonu}, \ref{lem:L4onv} and \ref{lem:L4onw} allow us to bound $\norm{u(\cdot,s)}_{L^p(\Omega)}$, $\norm{v(\cdot,s)}_{L^p(\Omega)}$, and $\norm{w(\cdot,s)}_{L^p(\Omega)})ds$ using $M_1(\tau)$, $N_3(\tau)$, and $N_4(\tau)$, respectively. We introduce $C_3(\tau)$ to denote the overall bound on these terms.
\begin{equation*}
    \begin{split}
         \norm{V_3}_{L^{\infty}(\Omega)}&\leq C_3(\tau)\int_0^t(t-s)^{-\theta_3}e^{-\nu_2(t-s)}ds\leq C_3(\tau)\int_0^\infty\sigma^{-\theta_3}e^{-\nu_2\sigma}d\sigma\\
         &\leq C_3(\tau)\frac{\Gamma(1-\theta_3)}{\nu_2^{1-\theta_3}}\leq C_3(\tau).
    \end{split}
\end{equation*}

We take the maximum of $C_1(\tau)$, $C_2(\tau)$, and $C_3(\tau)$ and obtain the bound on $\norm{v_1(\cdot,t)}_{L^\infty(\Omega)}$ (dependent on $\tau$ only).
\end{proof}

\begin{lemma}\label{lem:Linfonw}
From Lemma \ref{lem:L4onw}, let assumptions hold and constants be defined. Then for any $\tau\in(0,\min\{1,T_{\max}\})$, there exists constant $M_3(\tau)$ such that
      \begin{equation*}
          \norm{w(\cdot,t)}_{L^{\infty}(\Omega)}\leq M_3(\tau)\quad \forall t\in(\tau,T_{\max}).
      \end{equation*}
\end{lemma}

\begin{proof}
    The proof is identical to the proof for Lemma \ref{lem:Linfonv}.
\end{proof}
\begin{proof}(\textbf{Theorem \ref{thm:globex}}) Lemmas \ref{lem:W1infonu}, \ref{lem:Linfonv}, \ref{lem:Linfonw}, and part \textit{4} in Theorem \ref{thm:locex} provide the desired result. 
\end{proof}

\begin{remark}
    Note that $\tau$ is fixed and non-zero. It has been shown that in RAD systems, such as \eqref{eq:themainsys}, one can choose $\tau$ small enough so that the corresponding operator is a contraction map \cite{horstmann2005boundedness}. For $[0,\tau]$, a fixed-point theorem provides the bound, and for $(\tau,\infty)$, our proof provides the bound. Then, the larger bound between the two provides the needed uniformity. 
\end{remark}

\section{Linear stability analysis}\label{Sec:linst}
In this section, we investigate the fundamental dynamical properties of system \eqref{eq:themainsys}. Note that the general form of the trivial constant steady state is $(\overline{u},0,\overline{w})$, where 
\begin{equation*}
    \overline{w}=\frac{1}{|\Omega|}\int_\Omega w_0\; dx,
\end{equation*}
and $\overline{u}$ is a solution of the following equation
\begin{equation*}
    -sg_1(s,0,\overline{w})+h_1(s)=0.
\end{equation*}

For the remainder of the paper, we assume that $\overline{u}$ exists and analyze the linear stability of this trivial steady state. To do so, we first consider perturbations of the form
\begin{equation*}
    u=\overline{u}+\epsilon\widehat{u},\quad v=\epsilon\widehat{v},\quad w=\overline{w}+\epsilon\widehat{w},\quad0<\epsilon\ll 1.
\end{equation*}
Substituting those perturbations, linearizing around $\epsilon=0$, and collecting the terms attached to $\epsilon$ gives the following system for $(\widehat{u},\widehat{v},\widehat{w})$
\begin{equation*}
    \left\{
    \begin{aligned}
        \frac{\partial \widehat{u}}{\partial t}&=D_u\Delta \widehat{u} + (\widehat{u},\widehat{v},\widehat{w})\cdot \mathcal{U}  &&\text{ in } \Omega\times(0,T],\\
        \frac{\partial \widehat{v}}{\partial t}&=D_v \Delta \widehat{v}+ (\widehat{u},\widehat{v},\widehat{w})\cdot \mathcal{V} &&\text{ in } \Omega\times(0,T],\\
        \frac{\partial \widehat{w}}{\partial t}&=\nabla \cdot \big[D_w \nabla \widehat{w}-\overline{w} \chi_2(\overline{u}) \nabla \widehat{u} \big]&&\text{ in } \Omega\times(0,T],\\
         \frac{\partial u}{\partial \mathbf{n}}&= \frac{\partial v}{\partial \mathbf{n}}= \frac{\partial w}{\partial \mathbf{n}}=0&&\text{ on }\partial\Omega,
    \end{aligned}
    \right.
\end{equation*}
where
\begin{align*}
\mathcal{U}=\begin{bmatrix}
    -g_1(\overline{u},0,\overline{w})+\frac{d h_1(\overline{u})}{du}-\overline{u}\frac{\partial g_1(\overline{u},0,\overline{w})}{\partial u}\\[0.5em]
     \overline{u} \gamma f(\overline{u},0,\overline{w})-\overline{u}\frac{\partial g_1(\overline{u},0,\overline{w})}{\partial v}\\[0.5em]
    -\overline{u}\frac{\partial g_1(\overline{u},0,\overline{w})}{\partial w}
\end{bmatrix},
\mathcal{V}=\begin{bmatrix}
    0\\[0.5em]
    -\overline{u} f(\overline{u},0,\overline{w})-g_2(\overline{u},0,\overline{w})+h_2(0)\\[0.5em]
    0
\end{bmatrix}.
\end{align*}

Now we set $(\widehat{u},\widehat{v},\widehat{w})=(C_1,C_2,C_3)e^{\sigma t+i k\cdot x}$, where $\sigma$ is the growth rate and $k$ is the wavemode vector. We obtain the following eigenvalue problem

\begin{equation*}
    \sigma (C_1,C_2,C_3)^T=\mathcal{A}(C_1,C_2,C_3)^T,
\end{equation*}

where 
\begin{equation}\label{stabmatrix}
    \mathcal{A}=\begin{bmatrix}
    -|k|^2 D_u +\mathcal{U}_1&\mathcal{U}_2&\mathcal{U}_3\\[0.5em]
    0&-|k|^2D_v+\mathcal{V}_2&0\\[0.5em]
    |k|^2\overline{w}\chi_2(\overline{u})&0&-|k|^2D_w
    \end{bmatrix}.
\end{equation}

The stability of the trivial steady state is determined by the eigenvalues of \eqref{stabmatrix}, which has a row with all zero entries except for one. We obtain the following characteristic equation
\begin{equation*}
    (\sigma+|k|^2D_v-\mathcal{V}_2)\Big((\sigma+|k|^2D_u -\mathcal{U}_1)(\sigma +|k|^2D_w)-|k|^2\overline{w}\chi_2(\overline{u})\mathcal{U}_3\Big)=0.
\end{equation*}

One of the real roots is
\begin{equation*}
    \sigma_1=-|k|^2D_v +\mathcal{V}_2.
\end{equation*}

It is negative for any $k$ such that $|k|>0$, if $\mathcal{V}_2$ is negative
\begin{equation}\label{linstineq1}
\overline{u} f(\overline{u},0,\overline{w})+g_2(\overline{u},0,\overline{w})>h_2(0).
\end{equation}

The negativity of the real part of the other two roots can be checked via the Routh-Hurwitz criterion: all coefficients should be positive
\begin{align*}
    |k|^2(D_u+D_w)-\mathcal{U}_1>0\quad\text{and}\quad |k|^2\big(|k|^2D_uD_w-D_w\mathcal{U}_1-\overline{w}\chi_2(\overline{u})\mathcal{U}_3\big)>0.
\end{align*}

The inequalities are satisfied for any $k$ with $|k|>0$ if
\begin{align*}
    &\mathcal{U}_1<0 \iff    
g_1(\overline{u},0,\overline{w})+\overline{u}\frac{\partial g_1(\overline{u},0,\overline{w})}{\partial u}>
\frac{d h_1(\overline{u})}{du},\\
&\mathcal{U}_1<-\frac{\overline{w}\chi_2(\overline{u})\mathcal{U}_3}{D_w}\iff g_1(\overline{u},0,\overline{w})+\overline{u}\frac{\partial g_1(\overline{u},0,\overline{w})}{\partial u} +\frac{\overline{u}\overline{w}\chi_2(\overline{u})}{D_w}\frac{\partial g_1(\overline{u},0,\overline{w})}{\partial w} >\frac{d h_1(\overline{u})}{du},
\end{align*}

which can be combined into
\begin{align}\label{linstineq2}
g_1(\overline{u},0,\overline{w})+\overline{u}\frac{\partial g_1(\overline{u},0,\overline{w})}{\partial u} +\min\left\{\frac{\overline{u}\overline{w}\chi_2(\overline{u})}{D_w}\frac{\partial g_1(\overline{u},0,\overline{w})}{\partial w},\,0\right\} >\frac{d h_1(\overline{u})}{du}.
\end{align}

With the derived inequalities, we have the following proposition.

\begin{prop}\label{prop1}
    If inequalities \eqref{linstineq1} and \eqref{linstineq2} are satisfied, then the trivial constant steady state $(\overline{u},0,\overline{w})$ of system \eqref{eq:themainsys} is locally asymptotically stable. If at least one of the previous inequalities is false, then the trivial constant steady state is unstable. 
\end{prop}

Inequality \eqref{linstineq1} can hold provided that a removal of partakers is sufficient to overcome their intrinsic growth rate, and fails if the opposite is true. Inequality \eqref{linstineq2} can hold provided that a mitigating effect coming from the guardians is enough to outweigh the growth rate of the abstract value scalar field. In the real world, this means that if a guardian has enough power and freedom (\textit{i.e.}, they have control of $g_1$, $g_2$, and $w_0$), they may drive the system to the trivial steady state, effectively eradicating $v$ (\textit{e.g.}, the partaker density) and bringing $u$ (\textit{e.g.}, target's attractiveness) to the baseline level. Such an implication may help in deriving optimal strategies, in the sense of finding $g_1$, $g_2$, and $w_0$ that are adequate to suppress any potentially offensive activities. One can also note that, in general, $\frac{\partial g_1(\overline{u},0,\overline{w})}{\partial w}>0$, as the effect of the measures is proportional to the number of control agents. Hence, if $\chi_2(\overline{u})>0$ (\textit{e.g.}, on-hotspot movement of $w$), then only inequality \eqref{linstineq1} needs to be satisfied. It follows that, in this framework, off-hotspot (\textit{i.e.}, away from the high concentration regions) movement strategies require a larger effect from the $g_1$ function to maintain the stability of the trivial steady state. 

We do not consider other constant steady solutions, \textit{e.g.}, $\overline{v}>0$, as their existence depends if $\gamma uvf(u,v,w)- ug_1(u,v,w)+h_1(u)$ or $ - vuf(u,v,w)-vg_2(u,v,w)+vh_2(v)$ have non-negative real roots with respect to $u,v,w$. This, in turn, requires either putting additional assumptions or defining the functions explicitly. The procedure for analyzing the stability would be the same. However, the conditions may not be as explicit and/or simple as in Proposition \ref{prop1}, because matrix \ref{stabmatrix} may not have a nice structure for defining a determinant concisely; hence, finding the explicit roots of the corresponding characteristic equation could be more challenging. This is partially due to the complex formula for cubic equations. The issue can be circumvented by using the Routh-Hurwitz criterion instead; however, even this approach may be algebraically challenging in terms of obtaining explicit stability conditions, \textit{cf.} \cite{rodriguez2021understanding}.

\section{Applications and Numerical Simulations}\label{Sec:App}

In this section, we explore some potential applications of our theory, focusing on mathematical sociology. This area is relatively underexplored compared to the natural sciences. Reaction–advection–diffusion equations have been widely used in fields such as chemistry (chemotaxis), biology (morphogenesis), and ecology (species invasion) \cite{murray2011mathematical}. In the social sciences, mathematical modeling exists but is still developing. One reason is the inherent complexity of human behavior, which makes it challenging to choose the most appropriate framework for a given situation.  We study social dynamics using system \eqref{eq:themainsys} and consider two specific scenarios. The first model addresses the propensity for protest during mass public events. The second focuses on classroom bullying. Both are important and widespread social issues, yet formal mathematical models of these phenomena are relatively rare. 

\subsection{Protest propensity}

We consider an opinion-spreading model for protests or social uprisings, focusing on crowded public spaces where social tensions (\textit{i.e.}, state of unease between individuals and/or groups) are relatively high. We investigate how protesters (active or passive) contribute to the overall protesting intensity under certain control measures. Mathematical models of protest activity and social tension have been studied previously \cite{burbeck1978dynamics}, with extensions in \cite{berestycki2020modeling, bonnasse2018epidemiological}. 
These ``epidemiological" models have 
recently been applied to the 2019 Chilean riots \cite{caroca2020anatomy, cartes2022mathematical, cartes2025influence}. In this work, we examine how protest activity is shaped by new participants and different management strategies.

Let $A$ be a scalar field that measures the {\it propensity to protest}. In the sociological literature, the propensity to protest is defined as a person's disposition to think of participating in or to act by participating in protest activities \cite{brandstatter2014personality}.  The corresponding disposition is proportional to the present incentives. In particular, participation occurs when the difference between perceived benefit and cost passes a certain threshold. 
Note that recruiters can play an essential role in changing those perceptions. In a way, $A$ represents the concentration of the most active and vocal protestors, who may also serve as recruiters. 

There is also a ``curious" population $P$ on the streets, who may consider engaging in the protest. This group is attracted to high concentrations of $A$, as more information can be obtained in those spots. One may ask why the non-protesting population may be interested in potentially joining the protesting one. We assume that if people are not staying home and are on the street where the activity is happening, then at least they are somewhat partial. From a sociological perspective, this can be explained by low fear of repression, the absence of violence, and individual empathy \cite{aytacc2018people, wouters2019persuasive}. 
Changes in individual opinions are influenced by shifts in attitude within a respective local community, hence the presence of partiality \cite{xu2023sight}. Participation is proportional to the proximity (on a social network) of non-participating individuals to active movement members \cite{gonzalez2011dynamics}. 
It was also empirically observed that reinforcement from multiple sources is a crucial factor in obtaining the critical mass needed for shifting of protest opinions \cite{gonzalez2011dynamics}. This is akin to the broken windows effect, where repeated actions increase the susceptibility of victims. 

At the same time, there are individuals interested in reducing the propensity, for example, local police, protest managers, or peacekeepers. We denote the density of these agents by $M$. We assume that $M$ tends to move toward regions with high values of $A$, aiming to increase control over the protest. The influence of $M$ on both $P$ and $A$ depends strongly on context. One possible way to reduce propensity is through negotiation: open communication can help prevent violent outbreaks and may lead to peaceful dispersal \cite{nassauer2019}. By contrast, aggressive policing can have unintended effects, such as fueling further community radicalization \cite{ellefsen2021unintended}. In the next two subsections, we present two systems modeling these two different protest management strategies.  
\begin{equation}\label{eq:protest_general}
    \left\{
    \begin{aligned}
        \frac{\partial A}{\partial t}&=D_A\Delta A +\gamma AP\Big(f(A,P,M)-g_1(A,P.M)\Big)+\Phi_A- A, &&\text{ in } \Omega\times(0,T],\\
        \frac{\partial P}{\partial t}&=\nabla \cdot \left[D_P \nabla P-P\frac{\chi_P}{1+A}\nabla A \right]\\
        &+P\Big(- Af(A,P,M)-g_2(A,P,M)+\Phi_P-P\Big),&&\text{ in } \Omega\times(0,T],\\
        \frac{\partial M}{\partial t}&=\nabla \cdot \left[D_M \nabla M-M \frac{\chi_M}{1+A} \nabla A \right],&&\text{ in } \Omega\times(0,T],\\
            \frac{\partial A}{\partial \mathbf{n}}&= \frac{\partial P}{\partial \mathbf{n}}= \frac{\partial M}{\partial \mathbf{n}}=0,&&\text{ on }\partial\Omega,\\
A(\cdot&,0)=A_0\geq0,\,P(\cdot,0)=P_0\geq0,\,M(\cdot,0)=M_0\geq0,&&\text{ in }\Omega\times\{0\}.
    \end{aligned}
    \right.
\end{equation}

Both versions of the model are governed by system \eqref{eq:protest_general}. The diffusion terms represent the random movement of the agents. However, for $A$, the diffusion also accounts for the fact that propensity spreads locally due to the significant role of proximity in protest participation. The advection terms are saturated. Hence, $P$ and $M$ move along $\nabla A$ with rate $\chi_P P/(1+A)$ and $\chi_M M/(1+A)$, respectively. Note that the choice of 1 in the denominator is arbitrary and can be substituted with any positive number/parameter. Whenever a curious population encounters a highly active protest concentration, $AP$, they leave their compartment and, with an additional rate $\gamma=1$ (which remains fixed throughout), join the active protesters, thereby increasing the propensity. Both $A$ and $P$ have logistic growth rates with carrying capacities $\Phi_A$ and $\Phi_P$ respectively. The limit on local population sizes dictates logistic growth. The two models differ in reaction terms containing $M$, specifically in the terms $f,g_1,g_2$ in system \eqref{eq:protest_general}.

\subsubsection{Negotiating management}

In the first version, system \eqref{eq:protest_peace}, we assume that $M$ represents the density of peacekeeping protest management representatives. The only goal of $M$ agents is to reduce protest propensity as a means to avoid any violent outburst. Hence, we add the term $\frac{A P (\psi+M)}{1+e^{-(P-\Psi)}}$, which means that whenever all three scalar fields meet, the attractiveness reduces proportionally to the number of passive protesters and management. At the same time, this reaction term is ``activated" only when the number of protests at the location is sufficiently large ($\sim \Psi$). The parameter $\psi>0$ denotes the baseline management control. Note that $M$ does not target $P$ directly in this scenario, hence the absence of the reaction term between these two.  The model in this situation is given by 
\begin{equation}\label{eq:protest_peace}
    \left\{
    \begin{aligned}
        \frac{\partial A}{\partial t}&=D_A\Delta A +A\left(P-\frac{P (\psi+M)}{1+e^{-(P-\Psi)}}\right)+\Phi_A- A, &&\text{ in } \Omega\times(0,T],\\
        \frac{\partial P}{\partial t}&=\nabla \cdot \left[D_P \nabla P-P\frac{\chi_P}{1+A}\nabla A \right]+P\big(-A+\Phi_P-P\big),&&\text{ in } \Omega\times(0,T],\\
        \frac{\partial M}{\partial t}&=\nabla \cdot \left[D_M \nabla M-M \frac{\chi_M}{1+A} \nabla A \right],&&\text{ in } \Omega\times(0,T].
    \end{aligned}
    \right.
\end{equation}

\subsubsection{Enhanced policing}
In the second version, system \eqref{eq:protest_elevated}, we assume that $M$ represents the density of enhanced policing agents. Prior studies suggest that increases in policing may not always reduce unrest; in some contexts, they have been associated with heightened anger among protesters and stronger mobilization \cite{ellefsen2021unintended}. This escalation can contribute to greater social tensions and, in turn, a higher likelihood of protest activity. For instance, in 2019 in Iraq, police actions coincided with increased protest turnout \cite{curtice2021street}, and in the late 2010s in Uganda, strong repression was followed by greater dissent.  The model in this situation is given by 
\begin{equation}\label{eq:protest_elevated}
    \left\{
    \begin{aligned}
        \frac{\partial A}{\partial t}&=D_A\Delta A + \frac{AP}{1+P(\psi+M)}\\
        &+A\big(\tanh(PM-\Psi)\big)+\Phi_A-A, &&\text{ in } \Omega\times(0,T],\\
        \frac{\partial P}{\partial t}&=\nabla \cdot \left[D_P \nabla P-P\frac{\chi_P}{1+A}\nabla A \right]\\
        &+P\left(-\frac{A}{1+P(\psi+M)}- M+\Phi_P-P\right),&&\text{ in } \Omega\times(0,T],\\
        \frac{\partial M}{\partial t}&=\nabla \cdot \left[D_M \nabla M-M \frac{\chi_M}{1+A} \nabla A \right],&&\text{ in } \Omega\times(0,T].
    \end{aligned}
    \right.
\end{equation}

We assume that policing reduces both $A$ and $P$ through apprehensions and related interventions. This effect is represented in three ways. First, recruitment (the $AP$ term) is limited by the frequency of police encounters, modeled as saturation by $P(\psi+M)$, where $\psi>0$ represents a baseline level of inhibition. Second, enforcement may also deter or remove individuals who are not yet actively protesting, reflecting the broader reach of escalated repression \cite{ellefsen2021unintended}. Third, we include a backlash term, $A(\tanh(PM-\Psi))$, where removal beyond a threshold $\Psi$ increases social tensions, hence, protesting activity. In this way, policing can both suppress participation through deterrence and contribute to increased mobilization through backlash \cite{van2013emotions}. At first glance, system \eqref{eq:protest_elevated} may appear to be an extension of \eqref{eq:protest_peace}; however, there are noticeable differences, which are summarized in Table \ref{tab:summary} and will be more apparent in the following simulations.
 \begin{remark}
     One may note that system \eqref{eq:protest_elevated} does not satisfy (H5). However, $\tanh$ is bounded by 1, which means that the associated reaction term behaves linearly in $A$ and can be seen as a part of the growth rate. Hence, the proof of $W_{1,\infty}$-boundedness of $A$ holds. 
 \end{remark}

\begin{table}[H]
    \centering
\begin{tabular}{ | m{2.5cm} | m{6cm}| m{7cm} | } 
 \hline
 \cellcolor{gray!30} Reaction terms in \eqref{eq:protest_general}& \cellcolor{gray!20}  Negotiating management (\eqref{eq:protest_peace}) & \cellcolor{gray!20} Escalated forces, system \eqref{eq:protest_elevated}  \\ \hline
\cellcolor{gray!15}$ f(A,P,M)$& The expected number of events: $AP$& The saturation factor 
$1+P(\psi+M)$ increases with the number of escalated force actions against recruits, representing the resulting “chilling” effect.  \\ \hline
\cellcolor{gray!15} $g_1(A,P,M)$ & Is ``activated", when $P>\Psi$, where $\Psi$ is a tolerance of control agents towards the presence of newly joining protesters.& No activation occurs, but once $PM > \Psi$ (the demonstrators’ tolerance limit), backlash increases protest activity. The function effectively behaves as a growth rate since $P,M$ appear only inside $\tanh$. \\ \hline
 \cellcolor{gray!15}$g_2(A,P,M)$ & Does not directly reduce the presence of recruiters. & Decreases the rate of new participants by a factor of $-PM$.\\ \hline
\end{tabular}
 \caption{The difference in the reaction terms between systems \eqref{eq:protest_peace} and \eqref{eq:protest_elevated}. The conditions imposed on system \eqref{eq:protest_elevated} are more subtle. This is manifested not only in more involved functions, but also in the difference in tolerance parameter $\Psi$. The distinguishing property in the definition of this parameter will also be highlighted in the following simulations.}
    \label{tab:summary}
\end{table}

\subsubsection{Numerical Results}
In this section, we conduct numerical experiments to analyze the behavior of both systems and compare their solutions for suitable parameter values.

{\it Experimental set-up}. Two-dimensional simulations on square domains are carried out using MATLAB’s PDE Toolbox \cite{MATLAB, PdeToolbox}. To guide parameter selection, we use linear stability analysis of the trivial constant steady state, then choose values that produce non-trivial dynamics. Initial conditions are set so that the density of protesters begins at zero, while the other variables start near their steady-state values with small perturbations. We aim to understand the growth dynamics of the density of protesters, as well as the response of the other two quantities of interest. 

For each system configuration, we present two figures: one showing the evolution of the root mean square amplitude,
$$
f_{Amp}(t) := \sqrt{\tfrac{1}{|\Omega|}\int_\Omega f^2(x,y,t)dxdy},
$$
and another showing the final spatial patterns when relevant. Simulations are run for a longer period than the time windows illustrated in the plots, allowing the figures to capture both the emergence and stabilization of steady states. All three system components are included, with the exact parameters reported in the figure captions. 

{\it Simulations of the negotiated management model}. We start with the simulation of system \eqref{eq:protest_peace}. We want to understand how a change in a single parameter affects the solution behavior. In this case, we focus on tolerance $\Psi$, which is a threshold for when control measures are applied. In the first simulation, we use a relatively large $\Psi$, specifically $\Psi =5$. Figure \ref{fig:negot_amp_hetero} illustrates the emergence of spatially heterogeneous, yet constant-in-time, solutions of system \eqref{eq:protest_peace}. These solutions are sometimes referred to as ``hotspot''. The final pattern is shown in Figure \ref{fig:negot_end_hetero}, where a regular grid of circles is observed, with all three scalars concentrated. These hotspots emerge after some transient fluctuations. 

\begin{figure}[H]
    \centering
    \begin{subfigure}{0.4\textwidth}
    \includegraphics[width=\textwidth]{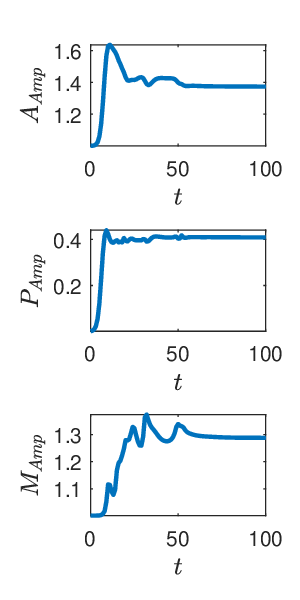}
    \caption{$\Psi=5$.}\label{fig:negot_amp_hetero}
    \end{subfigure}
    ~
    \begin{subfigure}{0.4\textwidth}
        \includegraphics[width=\textwidth]{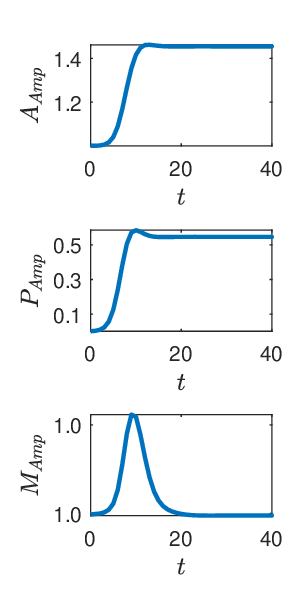}
        \caption{$\Psi=1$.}\label{fig:negot_amp_const}
    \end{subfigure}
    \caption{Root mean square of the solutions to system \eqref{eq:protest_peace}. The two subfigures show how a change in the single parameter $\Psi$ affects the solution behavior. In both \textbf{(a)} and \textbf{(b)}, the amplitudes become constant. Another common feature is that $A_{Amp}$ and $P_{Amp}$ rise and stabilize at a value higher than the initial.  However, the spatial landscape is different: in \textbf{(a)}, it is spatially heterogeneous, while in \textbf{(b)} it is uniform. The simulations are performed on a square domain with a side length of $\pi$ with parameters: $D_A=0.1$, $D_P=0.1$, $D_M=0.1$, $\chi_P=2$, $\chi_M=1$, $\Phi_A=1$, $\Phi_P=2$, $\psi=0.1$. $\Psi$ differs between the two subfigures. The initial conditions are chosen as $[A_0,P_0,M_0]+\varepsilon_0e^{-x-y}$, where $A_0=1.0$, $P_0=0.0$, $M_0=1.0$, and $\varepsilon_0=0.01$.  }
    \label{fig:negot_amp}
\end{figure}

One may ask if it is possible to suppress those hotspots, changing only the management-associated parameters. The answer is yes and no. When we apply Proposition \ref{prop1}, the trivial steady state is stable if $\Phi_A>\Phi_P$ ( $\overline{A}=\Psi_A$ in this case, as well). Hence, if $\Psi_A<\Psi_P$, no change in $\Psi$ would drive $P$ to 0, as the condition for inequality \eqref{linstineq1} is independent of $\Psi$. At the same time, there exists a different constant steady solution, as depicted in Figure \ref{fig:negot_amp_const}. This solution is achieved by decreasing $\Psi$, we take $\Psi=1$ in the figure. An interesting feature is that the overall trend in the evolution of amplitudes in both cases of $\Psi=5$ and $\Psi=1$ appears to be similar, at least for the $A$ and $A$ components, as shown in Figure \ref{fig:negot_amp}. However, in the latter case, the overall amplitude of the management agents returns to the initial constant solution, following an insignificant onset bump of $<0.001$. Moreover, one can see that in the hotspot case, the amplitudes of $A$ and $P$ are overall lower. This poses a dichotomy between control strategies: persistent high-value local hotspots or a uniform landscape with a larger amplitude. 

\begin{figure}[H]
    \centering
    \begin{subfigure}{1\textwidth}
    \includegraphics[width=\textwidth]{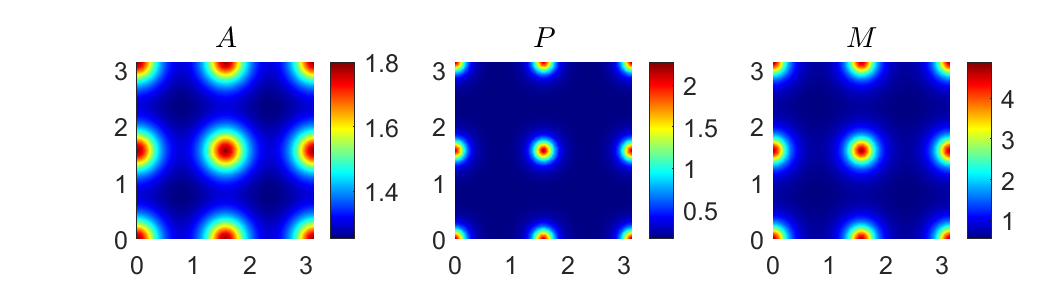}
    \caption{$\Psi=5$, $t=100$.}\label{fig:negot_end_hetero}
    \end{subfigure}
    \\
    \begin{subfigure}{1\textwidth}
        \includegraphics[width=\textwidth]{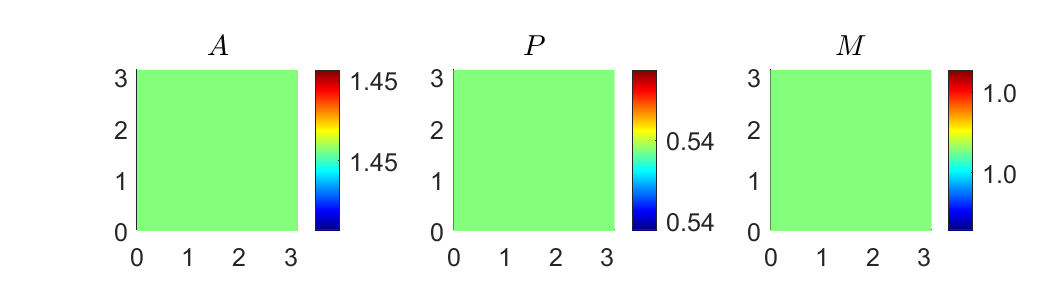}
        \caption{$\Psi=1$, $t=40$.}\label{fig:negot_end_const}
    \end{subfigure}
    \caption{The final time frames of the solutions to system \eqref{eq:protest_peace}. \textbf{(a)} Shows a grid-like pattern of circular clusters at the end of the simulation in Figure \ref{fig:negot_amp_hetero}. \textbf{(b)} Shows the corresponding simulation in Figure \ref{fig:negot_amp_const}, where decreasing $\Psi$ leads to a homogeneous landscape. The set-up and parameters are the same as in Figure \ref{fig:negot_amp}. } 
    \label{fig:negot_end}
\end{figure}

{\it Simulations of the enhanced policing model}. We now simulate solutions to system \eqref{eq:protest_elevated}, where $\Psi$ represents the tolerance of demonstrators toward police actions. Note that $\Psi$ has a different interpretation here compared to the previous section. Our goal is to understand how the decrease in $\Psi$ affects the dynamics of the system. The results are given in Figures \ref{fig:enhan_amp} and \ref{fig:enhan_end}. 

For large values of $\Psi$, the amplitudes stabilize and the landscape becomes homogeneous (Figure \ref{fig:enhan_amp_const}). As $\Psi$ decreases, the behavior becomes less predictable and appears periodic (Figure \ref{fig:enhan_amp_per}). At first glance, $A$ and $P$ seem to reach spatially heterogeneous but time-independent steady states. However, the amplitude field $M_{\text{Amp}}$ reveals small, rapid oscillations, most noticeable in the management component. Figure \ref{fig:enhan_end_end} shows that the components are not fully synchronized, though their ranges remain similar. The resulting pattern resembles a grid, with high-value regions shrinking progressively from $A$ to $M$. Management concentrates on protest hotspots, locally suppressing $A$ and $P$ and pushing them toward regions with weaker management presence. Consequently, large-scale amplitudes remain roughly constant (as in Figure \ref{fig:enhan_amp_per}), since the total mass does not change significantly. Instead, the dynamics manifest as continual cluster shifts. Additional solution frames in Figure \ref{fig:enhan_end_close} illustrate these displacements and make the time delay between the three components more apparent.

\begin{figure}[H]
    \centering
    \begin{subfigure}{0.45\textwidth}
    \includegraphics[width=\textwidth]{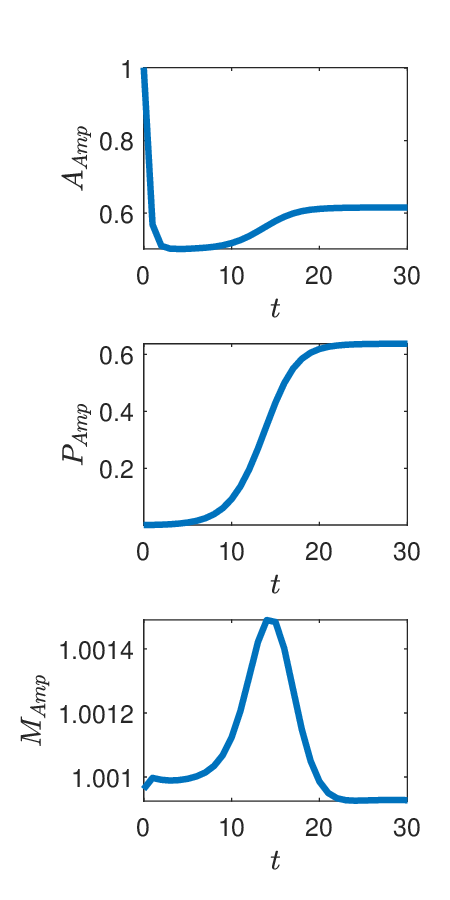}
    \caption{$\Psi=5$.}\label{fig:enhan_amp_const}
    \end{subfigure}
    ~
    \begin{subfigure}{0.45\textwidth}
        \includegraphics[width=\textwidth]{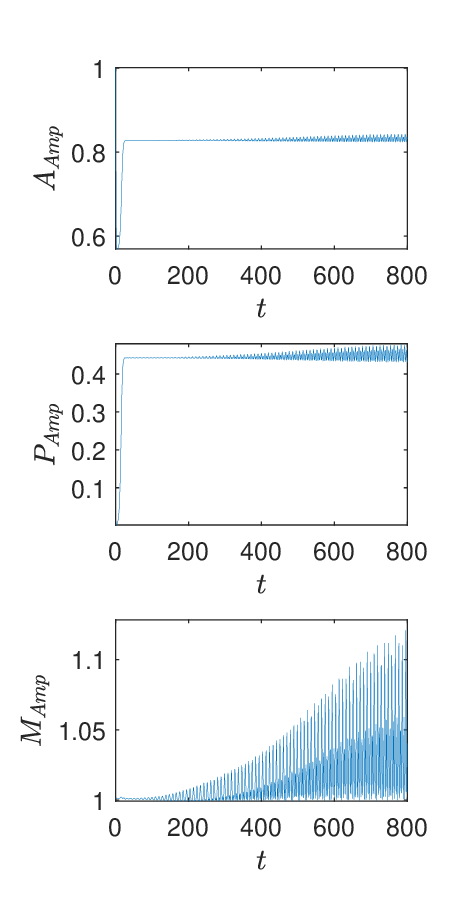}
        \caption{$\Psi=1$.}\label{fig:enhan_amp_per}
    \end{subfigure}
    \caption{Root mean square of the solutions to system \eqref{eq:protest_elevated}.  The two subfigures illustrate the effect of decreasing $\Psi$ on the system’s dynamics. In \textbf{(a)}, a higher $\Psi$ leads the system to a constant steady state, while in \textbf{(b)}, lowering $\Psi$ results in high-frequency periodic oscillations, demonstrating a transition from predictable to more complex behavior. 
The simulations are performed in the same setting as in Figure \ref{fig:negot_amp}. }
    \label{fig:enhan_amp}
\end{figure}
{\it Comparison of the models}. We compare the effect of the change of parameter $\Psi$ between two systems \eqref{eq:protest_peace} and \eqref{eq:protest_elevated}. Figure \ref{fig:enhan_amp} is a counterpart of Figure \ref{fig:negot_amp}, in the sense that the parameters between corresponding settings are the same.
Note that the behaviors are different; we achieve a non-trivial constant steady solution in Figure \ref{fig:enhan_amp_const} and a periodic solution in Figure \ref{fig:enhan_amp_per}, when compared to spatially heterogeneous and constant solutions in Figures \ref{fig:negot_amp_hetero} and \ref{fig:negot_amp_const}, respectively. Hence, a decrease in $\Psi$, which is a proxy for the ``tolerance" in both systems, leads to different behavior: in system \eqref{eq:protest_peace}, it leads to homogenization, whereas in system \eqref{eq:protest_elevated}, it leads to unpredictability. This highlights the importance of understanding the context, how management contributes to reducing protest propensity, and how participants react, to model the dynamics properly.
\begin{figure}[H]
    \centering
    \begin{subfigure}{1\textwidth}
    \includegraphics[width=\textwidth]{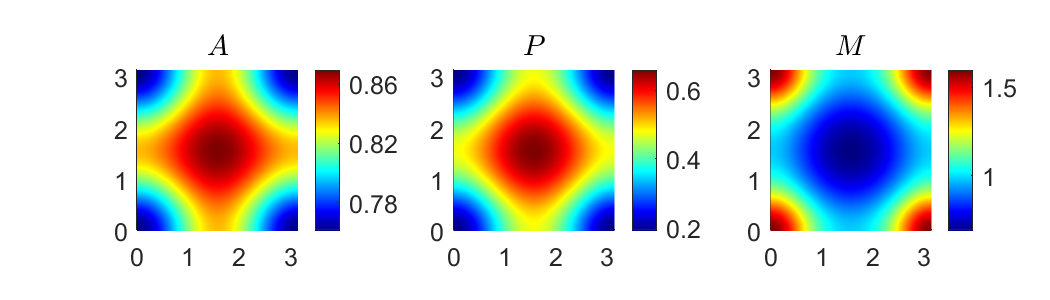}
    \caption{$t=1000$.}\label{fig:enhan_end_end}
    \end{subfigure}
    \\
    \begin{subfigure}{1\textwidth}
        \includegraphics[width=\textwidth]{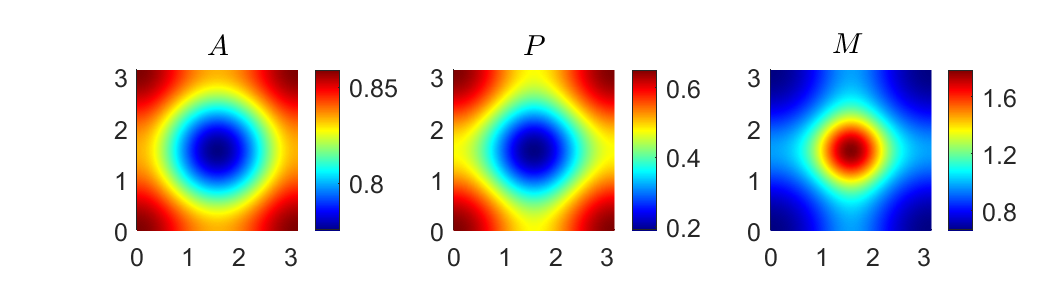}
        \caption{$t=995$.}\label{fig:enhan_end_close}
    \end{subfigure}
    \caption{Final frames of the solutions to system \eqref{eq:protest_elevated}. \textbf{(a)} Symmetric pattern formation where $A$ and $P$ are concentrated at the center and $M$ at the corners.
\textbf{(b)} A similar pattern with the clustering reversed. All components oscillate at the same frequency, but $M$ lags behind $A$ in its motion.
The setup and parameters are the same as in Figure \ref{fig:enhan_amp_per}.}
    \label{fig:enhan_end}
\end{figure}

{\it Key takeaways.} Our simulations indicate that prior knowledge of the reaction terms, determined by participants' actions and receptions of a situation, is crucial in modeling. This significance is corroborated by the observation that changes in these attitudes, as reflected in a decrease in the corresponding tolerance parameter, lead to entirely different solution behaviors. First, we see that reducing management's attitude towards protesters' actions, although it achieves homogenization, does not imply that overall protest activity is reduced. Second, we see that if the active protesters' tolerance for control force actions is reduced, it would lead to unpredictable behavior, which may be undesirable from the management perspective. In either case, and realistically, the tolerance parameters are dictated by the prior relationship between the community and management (be it administration, city council, \textit{etc}), hence our model simulations highlight the importance of maintaining proper communications between these two groups to avoid unpredictable and sometimes extreme behaviors.

\subsection{School bullying}
In this section, we model school bullying. The domain may range from a classroom to the whole school grounds.  There are a few existing models, mostly compartmental, that describe the dynamics in such situations, \textit{cf}, \cite{crokidakis2025mathematical, de2018modelling}. In those works, populations are divided into groups: bullies, victims, defenders, \textit{etc}. In this paper, we propose an alternative approach: rather than examining the sheer number of parties involved, we focus on the intensity of victimization. The reason for studying such an abstract notion is that negative social experiences are linked to a possible increase in aggressive tendencies \cite{reiter2018aggressive,sommer2014bullying}. A recent meta-analysis shows a partial correlation between the mental health of students and school violence \cite{polanin2021meta}. At the same time, the number of school shootings is steadily increasing each year \cite{rapa2024school}. 

Let $V$ be a scalar field representing the intensity of bullying, which can be seen as a proxy for identifying potential victims from the bully’s perspective. Studies show that bullying forms hotspots and can spread from one location to another, even within a single room \cite{migliaccio2017mapping}. Bullying often functions as a group process: nearby bystanders may join aggressors or themselves can become targets, depending on social networks and attitudes \cite{rambaran2020bullying}. This dynamic leads to the spread of harassment. It also reflects broader patterns of marginalization and social stratification \cite{peguero2012schools}. Moreover, bullying is an example of repeat and near-repeat victimization, much like residential burglaries \cite{farrell2014repeat}. Such repeated actions increase anxiety, isolation, and mental health problems, making individuals more vulnerable \cite{arseneault2018annual}. Together, these features suggest that $V$ follows dynamics similar to $A$ in the urban crime model \eqref{eq:E}.

The densities of bullies and guardians (which are usually called defenders) are denoted by the scalar fields $B$ and $G$, respectively. While bullies can be fellow students, the guardians can be teachers or vigilant peers. Both agents tend to go to places where potential victimization is highest. We assume that bullies leave the scene after the adverse action takes place or if a guardian is present. Guardians provide partial support, decreasing victimization. We note that bullying is a complicated (and perhaps, impossible) problem to fully eradicate, as it requires a complex approach \cite{salmivalli2021bullying}. We do not include other roles such as accomplices and witnesses. With the notion introduced previously, we have the following system
\begin{equation}\label{eq:bully}
    \left\{
    \begin{aligned}
        \frac{\partial V}{\partial t}&=D_V\Delta V +V\left(B(\Psi - B)-G\right)+\Phi_V-V, &&\text{ in } \Omega\times(0,T],\\
        \frac{\partial B}{\partial t}&=\nabla \cdot \left[D_B \nabla B-B\frac{\chi_B}{1+V}\nabla V \right]- BV\\ 
        &+B\big(-G(1+\tanh(V))+ \Phi_B-B \big), &&\text{ in } \Omega\times(0,T],\\
        \frac{\partial G}{\partial t}&=\nabla \cdot \left[D_G \nabla G-G \frac{\chi_G}{1+V} \nabla V \right],&&\text{ in } \Omega\times(0,T].
    \end{aligned}
    \right.
\end{equation}

System \eqref{eq:bully} presents one of the possible versions of the modeling. While some of the terms are self-explanatory and have been justified in the previous paragraphs, there are a few distinct features. First, note that the reaction between $V$ and $B$ is of the logistic type with respect to $B$. This can be explained by the fact that if there are too many bullies, potential targets may avoid common areas and even consider switching schools \cite{randa2019measuring}. In addition, victims may consider carrying a weapon, hence potentially reducing their (perceived) exposure to harassment \cite{esselmont2014carrying,van2014bullying}. In either of the cases, the number of victims or their vulnerability may be reduced if the school feels less safe because too many bullies are present. While the presence of a guardian may reduce victimization uniformly, we assume that bullies are removed by guardians only when the current level of victimization is sufficiently high. This is because students report that teachers may minimize or normalize bullying; therefore, there is a correlation between the perceived seriousness of bullying and the likelihood of intervention \cite{midgett2018rethinking}. Such dependence is represented through the already familiar $\tanh$ function that is shifted to preserve sign. Finally, to contrast with systems \eqref{eq:protest_peace} and \eqref{eq:protest_elevated}, we assume that there is no saturation in any of the terms. In the numerical simulation subsection, we will investigate whether the variation in terms (\textit{e.g.}, saturated or not) leads to different dynamics. 

\subsubsection{Numerical Results}

We return to the question of what can be done to change the solution behavior. The linear stability analysis highlights the importance of the initial distribution of the  $G_0$, which also corresponds to $\overline{G}$. In particular, there is an inverse relationship between $\overline{G}$ and $\overline{V}$ of the trivial steady state. We run numerical simulations to investigate the effect of changes in $\overline{G}$ on the system's overall behavior near the trivial steady state. We follow the same steps in the numerical analysis as we did for the protest propensity models.

\begin{figure}[H]
    \centering
    \begin{subfigure}{0.3\textwidth}
    \includegraphics[width=\textwidth]{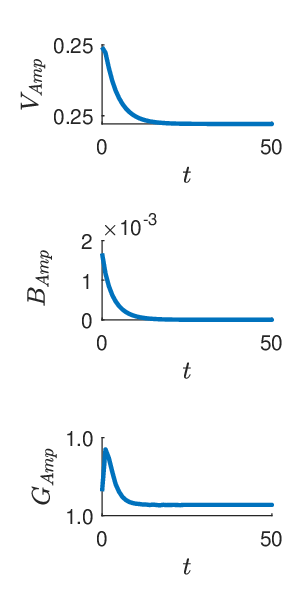}
    \caption{$V_0=0.25,\,G_0=1.0$.}\label{fig:bully_amp_const_1}
    \end{subfigure}
    ~
    \begin{subfigure}{0.3\textwidth}
        \includegraphics[width=\textwidth]{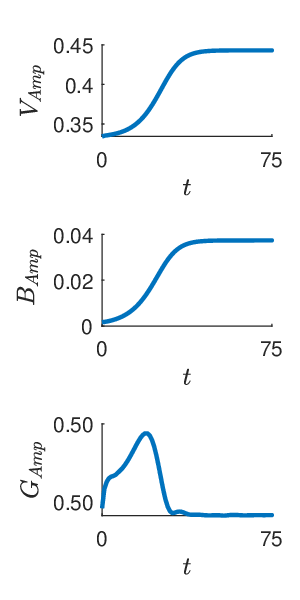}
\caption{$V_0=0.333,\,G_0=0.5$.}\label{fig:bully_amp_const_2}

    \end{subfigure}
        ~
    \begin{subfigure}{0.3\textwidth}
        \includegraphics[width=\textwidth]{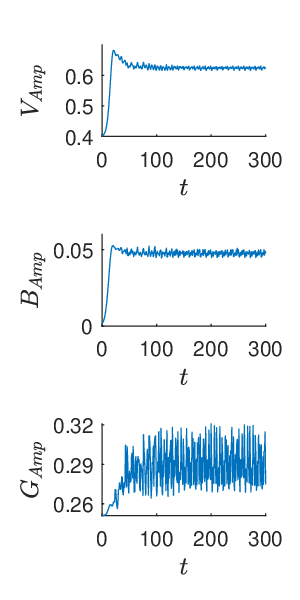}
\caption{$V_0=0.4,\,G_0=0.25$.}\label{fig:bully_amp_per_1}

    \end{subfigure}
    \caption{Root mean square of the solutions to system \eqref{eq:protest_elevated}. The three subfigures illustrate how changes in the initial density of the guardian agents affect solution behavior, particularly the bullies' steady-state density. Consecutive decrease in $G_0$ and corresponding increase in $V_0$ changes the steady state solution behavior from \textbf{(a)} trivial to \textbf{(b)} non-trivial constant to \textbf{(c)} periodic. $B_0=0$ in every settings, yet average $B$ increases from \textbf{(a)}, where it is zero to \textbf{(c)}, where it fluctuates around $0.05$. The domain of the simulation is a square with a side length of $\pi$. The parameters: $D_V=0.05$, $D_B=0.05$, $D_G=0.05$, $\chi_B=2$, $\chi_G=2$, $\Phi_G=0.5$, $\Phi_B=1$, $\Psi=10$. The initial conditions are chosen as $[V_0,B_0,G_0]+\varepsilon_0e^{-x-y}$, where $B_0=0.0$,  $\varepsilon_0=0.01$, and $V_0, G_0$ varies between subfigures.}
    \label{fig:bully_amp}
\end{figure}

Figure \ref{fig:bully_amp} illustrates the effect of decreasing $\overline{G}$. As $\overline{u}$ of the trivial steady state is dependent on $\overline{G}$, we also adjust it. The remaining parameters are the same between the three cases. Between Figures \ref{fig:bully_amp_const_1} and \ref{fig:bully_amp_const_2}, we halve the number of guardians $1.0\to0.5$. This causes the trivial steady state to become unstable, and a new constant steady state solution emerges. We omit the final frames, as no interesting pattern has formed. 
Further halving of $\overline{G}$, $0.5\to0.25$, between Figures \ref{fig:bully_amp_const_2} and \ref{fig:bully_amp_per_1}, leads to the emergence of the periodic behavior and the disappearance of a constant steady-state solution. The amplitudes of the victimization and the bullies' density are also larger than before.  Another important detail is that the transients pass more quickly with an increase in $\overline{G}$. The evolution of amplitude stabilizes at around $t=25$, $=50$, and $=150$ in Figure \ref{fig:bully_amp}, respectively. This may be important when considering time length limitations in achieving the needed values and the landscape of victimization. Periodic behavior generates an interesting pattern, which is commented on in the following paragraph.

\begin{figure}[H]
    \centering
    \begin{subfigure}{1\textwidth}
    \includegraphics[width=\textwidth]{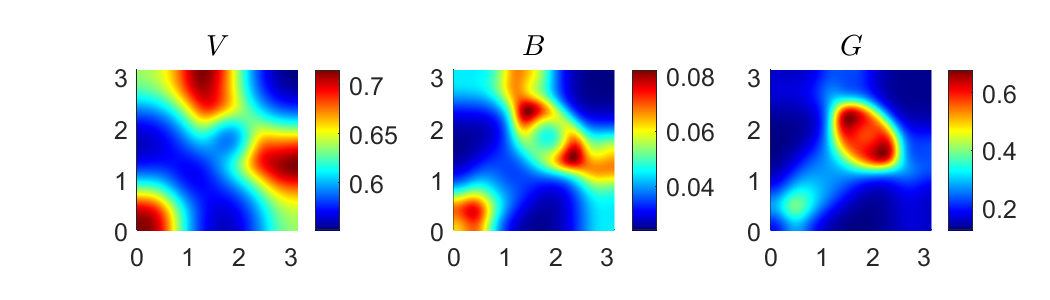}
    \caption{$t=499$.}\label{fig:bully_end_close}
    \end{subfigure}
    \\
    \begin{subfigure}{1\textwidth}
        \includegraphics[width=\textwidth]{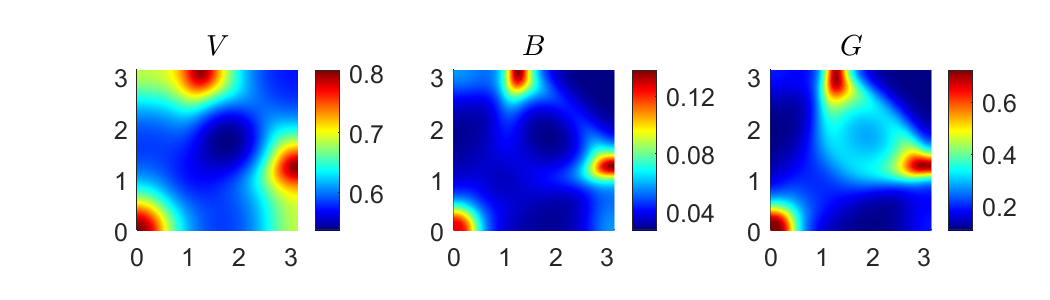}
        \caption{$t=500$.}\label{fig:bully_end_end}
    \end{subfigure}
    \caption{The final time frames of the solutions to system \eqref{eq:protest_elevated}.
From \textbf{(a)} to \textbf{(b)}, clusters of $B$ and $V$ move toward the boundaries, seemingly pushed by $G$. The difference between the two frames is only one time unit, yet $G$ moves quickly enough to localize at the new concentrations of $B$ and $V$. Despite this rapid and precise movement, it is not sufficient to reduce the densities of victimization and bullying to zero, as shown in Figure \ref{fig:bully_amp_per_1}. The setup and parameters are the same as in Figure \ref{fig:bully_amp_per_1}.
}
    \label{fig:bully_end}
\end{figure}

Periodic behavior can be observed in Figure \ref{fig:bully_amp_per_1}. The initial phases look aperiodic, but on the horizon, the solutions seem to stabilize. The changes are equally spaced in time, but not necessarily of the same amplitude. As before, the third panel of this figure provides more evidence of this behavior. The overall amplitude of bully density is on relatively minor scales. In Figure  \ref{fig:bully_end_end}, the second panel shows the maximum value is $0.05$, which is not far from the initial $0.01$. The maximum value of the guardians appears in the exact location as the maximum value of victimization and bullies' density: the components' clusters follow each other. The dynamics appear to be similar to those of \eqref{eq:protest_elevated}. It may seem that scalar fields are synchronized, but the second-to-last solution frames, in Figure \ref{fig:bully_end_close}, show that there is a delay in the evolution of concentrations. Plainly, high values of $V$ are ``followed" by high values of $B$, and high values of $B$ are ``followed" by high values of $G$. $G$ in turn, eradicates $V$ and $G$ locally pushing them to other regions. The periodicity of the amplitudes implies a cycle of such high-value concentration movements.

{\it Key takeaways.} The simulation of the proposed model for bullying highlighted the importance of having a sufficient initial number of guardians or partial people in general to homogenize victimization and reduce bullying density. On the other hand, if the number of guardians is inadequate, the system exhibits unpredictable behavior, making it more challenging to manage harassment, given limited school resources. That is not to say that a sufficient number of guardians solves the bullying problem entirely, or that there is no solution to this issue in a limited-resources setup. Instead, this observation highlights the importance of exploring other avenues (\textit{e.g.}, changes in other parameters) and designing complex approaches that, for example, include introducing temporary damping forces, which in the real world can be manifested through occasional events that raise awareness. 

\subsection{Additional comments on the models}
As a final note, the observations of the importance of $\overline{w}$ made above are true irrespective of the system due to the universality of Proposition \ref{prop1}. In the settings we considered, inequality \eqref{linstineq2} is satisfied, so the linear stability is dependent on inequality \eqref{linstineq1}. There, the left-hand side should be large enough to outweigh $\Phi_B$. On the other hand, a root to $-sg_1(s,0,\overline{w})+h_1(s)$ is inversely proportional to $\overline{w}$. This highlights the significance of the initial mass of guardian agents: not only may it be a factor in determining whether the trivial steady state can be achieved, but it also governs the value of $\overline{u}$ at that state. This is somewhat intuitive, but in this paper, we can have empirical confirmation. 

Although implications from the theory and simulations may seem helpful in informing policies, one must stay critical in interpreting the results. Systems \eqref{eq:protest_peace} and \eqref{eq:protest_elevated} may require further refinement to make them more realistic. For example, in the urban crime model, the advection is of the form $v\chi_1\nabla\ln(u)$, which is not bounded or, in some versions of predator-prey models, the reactions are not saturated \cite{short2008statistical}. A further step would be to incorporate spatial heterogeneity or time dependence into any of the terms. For instance, make parameters fluctuate to represent day and night protest activities, or overlay a graph that represents special means of communication or transportation (\textit{cf}, a subway network integration in the 2019 Chilean riots study \cite{cartes2022mathematical}). 
Another extension would be the introduction of randomness and stochasticity, especially in the reaction terms between target and actor agents. 

The same applies to system \eqref{eq:bully}, where we have omitted factors such as students who are both bullies and victims. Additionally, teachers (who serve as guardians) may also contribute to $V$ positively, as they can be either bullies or victims themselves \cite{twemlow2006teachers}. However, note that any changes from \eqref{eq:themainsys} would make the analysis more convoluted, and any deviation from hypotheses would mean that Theorems \ref{thm:locex} and \ref{thm:globex} may not apply. We still hope that our models can serve as a starting point for all these extensions.

\subsection{Other possible scenarios}
We have covered two settings in detail, and even within these two, some variations exist that make models more specific. For example, one can consider a cyber-attack by viewing a website as a domain, and locality is defined through the networks, such as mutual friends and followers. The malicious acts could be bullying, fake news spreading, \textit{etc.} The underlying feature is the presence of repeat victimization or a similar phenomenon \cite{moneva2022repeat}. There are multiple reasons why a user's vulnerability increases with the frequency of attacks, including the illusory truth effect, the echo chamber effect, and polarization and partisanship \cite{cinelli2021echo, pennycook2021psychology,tandoc2019facts}. Similar concepts (akin to repeat and near-repeat victimization and broken windows effect) can be found in offline scenarios such as scamming \cite{fischer2013individuals, hanoch2021scams}, proselytizing \cite{bainbridge1979cult}, and rumoring \cite{berinsky2023political, zhu2024rumours}.  
No matter the setting, they all share the same framework: the existence of an abstract scalar field and two agents who increase or decrease the abstract scalar field, respectively. Instead of continuing with an extensive list of potential applications, we proceed with numerical simulations of systems \eqref{eq:protest_peace}, \eqref{eq:protest_elevated}, and \eqref{eq:bully} to demonstrate the behavior of the solution. 

\section{Conclusion}

In this paper, we have proposed a general model that has the potential to be widely applied in mathematical sociology. Motivated by the scope of theoretical and practical findings of classical predator-prey and relatively new urban crime models, our RAD system is based on routine activity theory, \textit{i.e.}, we model the dynamics between target, partaker, and guardian. Our model is sufficiently general to accommodate various types of relationships, \textit{e.g.}, Holling type II. Under certain assumptions, we proved the local existence and relevant regularities, and subsequently, global existence ensued. Our proofs are based on results from Amman's parabolic and operator semigroup theories. We performed the linear stability analysis of the trivial steady state. We then demonstrated the potential applications of our model by defining specific functions and basing them on corresponding literature. 

We applied our system to model the dynamics of protest and school bullying. We attempted to capture the effects of the partaker and controller on their respective abstract fields through nonlinear functions. Within the protest scenario, we considered two cases of how management can behave: more peaceful negotiators and an enhanced police force. When we numerically simulated the models, we observed that the systems may exhibit various types of behavior ranging from trivial to chaotic. We explored what would happen to the solutions if we modified the parameters that are linked to the control units. We see that a sufficient change in one parameter may significantly alter the behavior. The simulation results appear to agree with linear stability analysis and the general intuition behind the real-world phenomena. 

The primary goal of this paper is to introduce the model to both mathematical and sociological communities; hence, there are many directions to follow. First, one can further investigate the dynamical properties of the general system: pattern formation, bifurcating branches, \textit{etc}, \textit{cf.} \cite{rodriguez2021understanding,yerlanov2026}.  Second, both application models need additional confirmation and refinement. 
Once these processes are complete, one can address practical questions, such as exploring optimal harm reduction, stabilizing the system, and hotspot suppression, among others \cite{yerlanov2025}. Finally, after tuning to a specific setting, the data fitting can be performed with the hopes of better understanding and predicting community patterns. 

This expository paper should be viewed as a part of a bigger effort to expand the domains of mathematical applications and motivate new interdisciplinary collaborations to tackle complex, critical issues. Our models are insufficient to address the social problems on their own. Instead, these models should be incorporated into bigger research designs with a multiscale approach. We hope that readers will feel inspired to contribute to the field of mathematical sociology.

{\bf Acknowledgement:}  This work was partially funded by NSF-DMS-2042413 and AFOSR MURI FA9550-22-1-0380.

\bibliographystyle{plain} 
\bibliography{ref}

\begin{thebibliography}{10}

\bibitem{alikakos1979lp}
Nicholas~D Alikakos.
\newblock ${L}^p$ bounds of solutions of reaction-diffusion equations.
\newblock {\em Communications in Partial Differential Equations}, 4(8):827--868, 1979.

\bibitem{amann_dynamic_1989}
Herbert Amann.
\newblock Dynamic theory of quasilinear parabolic systems. {III}. {Global} existence.
\newblock {\em Mathematische Zeitschrift}, 202(2):219--250, June 1989.

\bibitem{amann_dynamic_1990}
Herbert Amann.
\newblock Dynamic theory of quasilinear parabolic equations. {II}. {Reaction}-diffusion systems.
\newblock {\em Differential and Integral Equations}, 3(1):13--75, January 1990.

\bibitem{amann_nonhomogeneous_1993}
Herbert Amann.
\newblock Nonhomogeneous linear and quasilinear elliptic and parabolic boundary value problems.
\newblock In {\em Function spaces, differential operators and nonlinear analysis}, pages 9--126. Vieweg+Teubner Verlag, Wiesbaden, 1993.

\bibitem{anselin2000spatial}
Luc Anselin, Jacqueline Cohen, David Cook, Wilpen Gorr, and George Tita.
\newblock Spatial analyses of crime.
\newblock {\em Criminal Justice}, 4(2):213--262, 2000.

\bibitem{arseneault2018annual}
Louise Arseneault.
\newblock Annual {R}esearch {R}eview: The persistent and pervasive impact of being bullied in childhood and adolescence: implications for policy and practice.
\newblock {\em Journal of Child Psychology and Psychiatry}, 59(4):405--421, 2018.

\bibitem{aytacc2018people}
S~Erdem Ayta{\c{c}}, Luis Schiumerini, and Susan Stokes.
\newblock Why do people join backlash protests? {L}essons from {T}urkey.
\newblock {\em Journal of Conflict Resolution}, 62(6):1205--1228, 2018.

\bibitem{bainbridge1979cult}
William~Sims Bainbridge and Rodney Stark.
\newblock Cult formation: Three compatible models.
\newblock {\em Sociological Analysis}, 40(4):283--295, 1979.

\bibitem{berestycki2020modeling}
Henri Berestycki, Samuel Nordmann, and Luca Rossi.
\newblock Modeling the propagation of riots, collective behaviors, and epidemics.
\newblock preprint can be found at \url{https://arxiv.org/abs/2005.09865}, 2020.

\bibitem{berinsky2023political}
Adam~J Berinsky.
\newblock Rumours in the political world.
\newblock In {\em Political rumors: Why we accept misinformation and how to fight it}, pages 36--57. Princeton University Press, 2023.

\bibitem{bonnasse2018epidemiological}
Laurent Bonnasse-Gahot, Henri Berestycki, Marie-Aude Depuiset, Mirta~B Gordon, Sebastian Roch{\'e}, Nancy Rodr{\'\i}guez, and Jean-Pierre Nadal.
\newblock Epidemiological modelling of the 2005 {F}rench riots: {A} spreading wave and the role of contagion.
\newblock {\em Scientific Reports}, 8(1):107, 2018.

\bibitem{bowers2005domestic}
Kate~J Bowers and Shane~D Johnson.
\newblock Domestic burglary repeats and space-time clusters: The dimensions of risk.
\newblock {\em European Journal of Criminology}, 2(1):67--92, 2005.

\bibitem{braga2015can}
Anthony~A Braga, Brandon~C Welsh, and Cory Schnell.
\newblock Can policing disorder reduce crime? {A} systematic review and meta-analysis.
\newblock {\em Journal of Research in Crime and Delinquency}, 52(4):567--588, 2015.

\bibitem{brandstatter2014personality}
Hermann Brandst{\"a}tter and Karl-Dieter Opp.
\newblock Personality traits (“{B}ig five”) and the propensity to political protest: Alternative models.
\newblock {\em Political Psychology}, 35(4):515--537, 2014.

\bibitem{burbeck1978dynamics}
Stephen~L Burbeck, Walter~J Raine, and MJ~Abudu Stark.
\newblock The dynamics of riot growth: An epidemiological approach.
\newblock {\em Journal of Mathematical Sociology}, 6(1):1--22, 1978.

\bibitem{cantrell2004spatial}
Robert~S Cantrell and Chris Cosner.
\newblock {\em Spatial ecology via reaction-diffusion equations}.
\newblock John Wiley \& Sons, 2004.

\bibitem{caroca2020anatomy}
Paulina Caroca~Soto, Carlos Cartes, Toby~P Davies, Jocelyn Olivari, Sergio Rica, and Katia Vogt-Geisse.
\newblock The anatomy of the 2019 {C}hilean social unrest.
\newblock {\em Chaos: An Interdisciplinary Journal of Nonlinear Science}, 30(7), 2020.

\bibitem{cartes2022mathematical}
Carlos Cartes.
\newblock Mathematical modeling of the {C}hilean riots of 2019: An epidemiological non-local approach.
\newblock {\em Chaos: An Interdisciplinary Journal of Nonlinear Science}, 32(12), 2022.

\bibitem{cartes2025influence}
Carlos Cartes.
\newblock Influence of commuter rioters and income distribution on the 2019 {C}hilean unrest.
\newblock {\em Chaos: An Interdisciplinary Journal of Nonlinear Science}, 35(5), 2025.

\bibitem{cinelli2021echo}
Matteo Cinelli, Gianmarco De~Francisci~Morales, Alessandro Galeazzi, Walter Quattrociocchi, and Michele Starnini.
\newblock The echo chamber effect on social media.
\newblock {\em Proceedings of the National Academy of Sciences}, 118(9):e2023301118, 2021.

\bibitem{cosner2008reaction}
Chris Cosner.
\newblock Reaction--diffusion equations and ecological modeling.
\newblock In Avner Friedman, editor, {\em Tutorials in mathematical biosciences IV: Evolution and ecology}, pages 77--115. Springer, 2008.

\bibitem{crokidakis2025mathematical}
Nuno Crokidakis.
\newblock A mathematical model for the bullying dynamics in schools.
\newblock {\em Applied Mathematics and Computation}, 492:129254, 2025.

\bibitem{curtice2021street}
Travis~B Curtice and Brandon Behlendorf.
\newblock Street-level repression: {P}rotest, policing, and dissent in uganda.
\newblock {\em Journal of Conflict Resolution}, 65(1):166--194, 2021.

\bibitem{de2018modelling}
Elena De~la Poza, Lucas J{\'o}dar, and Luc{\'\i}a Ram{\'\i}rez.
\newblock Modelling bullying propagation in spain: {A} quantitative and qualitative approach.
\newblock {\em Quality \& Quantity}, 52(4):1627--1642, 2018.

\bibitem{ellefsen2021unintended}
Rune Ellefsen.
\newblock The unintended consequences of escalated repression.
\newblock {\em Mobilization: An International Quarterly}, 26(1):87--108, 2021.

\bibitem{esselmont2014carrying}
Chris Esselmont.
\newblock Carrying a weapon to school: {T}he roles of bullying victimization and perceived safety.
\newblock {\em Deviant Behavior}, 35(3):215--232, 2014.

\bibitem{farrell2014repeat}
Graham Farrell and Ken Pease.
\newblock Repeat victimization.
\newblock In Gerben Bruinsma and David Weisburd, editors, {\em Encyclopedia of criminology and criminal justice}, pages 4371--4381. Springer, 2014.

\bibitem{felson2017routine}
Marcus Felson.
\newblock Routine activities and crime prevention in the developing metropolis.
\newblock {\em Criminology}, 25:911--932, 1987.

\bibitem{fischer2013individuals}
Peter Fischer, Stephen~EG Lea, and Kath~M Evans.
\newblock Why do individuals respond to fraudulent scam communications and lose money? {T}he psychological determinants of scam compliance.
\newblock {\em Journal of Applied Social Psychology}, 43(10):2060--2072, 2013.

\bibitem{gharasoo2014chemotactic}
Mehdi Gharasoo, Florian Centler, Ingo Fetzer, and Martin Thullner.
\newblock How the chemotactic characteristics of bacteria can determine their population patterns.
\newblock {\em Soil Biology and Biochemistry}, 69:346--358, 2014.

\bibitem{gonzalez2011dynamics}
Sandra Gonz{\'a}lez-Bail{\'o}n, Javier Borge-Holthoefer, Alejandro Rivero, and Yamir Moreno.
\newblock The dynamics of protest recruitment through an online network.
\newblock {\em Scientific Reports}, 1(1):1--7, 2011.

\bibitem{hanoch2021scams}
Yaniv Hanoch and Stacey Wood.
\newblock The scams among us: {W}ho falls prey and why.
\newblock {\em Current Directions in Psychological Science}, 30(3):260--266, 2021.

\bibitem{horstmann2005boundedness}
Dirk Horstmann and Michael Winkler.
\newblock Boundedness vs. blow-up in a chemotaxis system.
\newblock {\em Journal of Differential Equations}, 215(1):52--107, 2005.

\bibitem{johnson1997new}
Shane~D Johnson, Kate Bowers, and Alex Hirschfield.
\newblock New insights into the spatial and temporal distribution of repeat victimization.
\newblock {\em British Journal of Criminology}, 37(2):224--241, 1997.

\bibitem{jones2010statistical}
Paul~A Jones, P~Jeffrey Brantingham, and Lincoln~R Chayes.
\newblock Statistical models of criminal behavior: {T}he effects of law enforcement actions.
\newblock {\em Mathematical Models and Methods in Applied Sciences}, 20(supp01):1397--1423, 2010.

\bibitem{lam2022introduction}
King-Yeung Lam and Yuan Lou.
\newblock {\em Introduction to reaction-diffusion equations: {T}heory and applications to spatial ecology and evolutionary biology}.
\newblock Springer Nature, 2022.

\bibitem{midgett2018rethinking}
Aida Midgett, Diana~M Doumas, April Johnston, Rhiannon Trull, and Raissa Miller.
\newblock Rethinking bullying interventions for high school students: {A} qualitative study.
\newblock {\em Journal of Child and Adolescent Counseling}, 4(2):146--163, 2018.

\bibitem{migliaccio2017mapping}
Todd Migliaccio, Juliana Raskauskas, and Mathew Schmidtlein.
\newblock Mapping the landscapes of bullying.
\newblock {\em Learning Environments Research}, 20(3):365--382, 2017.

\bibitem{miro2014routine}
Fernando Mir{\'o}.
\newblock Routine activity theory.
\newblock In J.~Mitchell Miller, editor, {\em The encyclopedia of theoretical criminology}, pages 1--7. Wiley Online Library, 2014.

\bibitem{moneva2022repeat}
Asier Moneva, E~Rutger Leukfeldt, Steve~GA Van De~Weijer, and Fernando Mir{\'o}-Llinares.
\newblock Repeat victimization by website defacement: {A}n empirical test of premises from an environmental criminology perspective.
\newblock {\em Computers in Human Behavior}, 126:106984, 2022.

\bibitem{murray2011mathematical}
J.D. Murray.
\newblock {\em Mathematical biology II: Spatial models and biomedical applications}.
\newblock Interdisciplinary Applied Mathematics. Springer New York, 2011.

\bibitem{nassauer2019}
Anne Nassauer.
\newblock {\em How to Keep Protests Peaceful}.
\newblock Oxford University Press, 09 2019.

\bibitem{nirenberg1959elliptic}
Louis Nirenberg.
\newblock On elliptic partial differential equations.
\newblock {\em Annali della Scuola Normale Superiore di Pisa-Scienze Fisiche e Matematiche}, 13(2):115--162, 1959.

\bibitem{peguero2012schools}
Anthony~A Peguero.
\newblock Schools, bullying, and inequality: {I}ntersecting factors and complexities with the stratification of youth victimization at school.
\newblock {\em Sociology Compass}, 6(5):402--412, 2012.

\bibitem{pennycook2021psychology}
Gordon Pennycook and David~G Rand.
\newblock The psychology of fake news.
\newblock {\em Trends in Cognitive Sciences}, 25(5):388--402, 2021.

\bibitem{pitcher2010adding}
Ashley~B Pitcher.
\newblock Adding police to a mathematical model of burglary.
\newblock {\em European Journal of Applied Mathematics}, 21(4-5):401--419, 2010.

\bibitem{polanin2021meta}
Joshua~R Polanin, Dorothy~L Espelage, Jennifer~K Grotpeter, Elizabeth Spinney, Katherine~M Ingram, Alberto Valido, America El~Sheikh, Cagil Torgal, and Luz Robinson.
\newblock A meta-analysis of longitudinal partial correlations between school violence and mental health, school performance, and criminal or delinquent acts.
\newblock {\em Psychological Bulletin}, 147(2):115, 2021.

\bibitem{qiu2023dynamics}
Shuyan Qiu, Chunlai Mu, and Xinyu Tu.
\newblock Dynamics for a three-species predator-prey model with density-dependent motilities.
\newblock {\em Journal of Dynamics and Differential Equations}, 35(1):709--733, 2023.

\bibitem{rambaran2020bullying}
J~Ashwin Rambaran, Jan~Kornelis Dijkstra, and Ren{\'e} Veenstra.
\newblock Bullying as a group process in childhood: {A} longitudinal social network analysis.
\newblock {\em Child Development}, 91(4):1336--1352, 2020.

\bibitem{randa2019measuring}
Ryan Randa, Bradford~W Reyns, and Matt~R Nobles.
\newblock Measuring the effects of limited and persistent school bullying victimization: {R}epeat victimization, fear, and adaptive behaviors.
\newblock {\em Journal of Interpersonal Violence}, 34(2):392--415, 2019.

\bibitem{rapa2024school}
Luke~J Rapa, Antonis Katsiyannis, Samantha~N Scott, and Olivia Durham.
\newblock School shootings in the {U}nited {S}tates: 1997--2022.
\newblock {\em Pediatrics}, 153(4):e2023064311, 2024.

\bibitem{reiter2018aggressive}
Katharina Reiter-Scheidl, Ilona Papousek, Helmut~K Lackner, Manuela Paechter, Elisabeth~M Weiss, and Nil{\"u}fer Aydin.
\newblock Aggressive behavior after social exclusion is linked with the spontaneous initiation of more action-oriented coping immediately following the exclusion episode.
\newblock {\em Physiology \& Behavior}, 195:142--150, 2018.

\bibitem{rodriguez2021understanding}
Nancy Rodr{\'\i}guez, Qi~Wang, and Lu~Zhang.
\newblock Understanding the effects of on-and off-hotspot policing: {E}vidence of hotspot, oscillating, and chaotic activities.
\newblock {\em SIAM Journal On Applied Dynamical Systems}, 20(4):1882--1916, 2021.

\bibitem{salmivalli2021bullying}
Christina Salmivalli, Lydia Laninga-Wijnen, Sarah~T Malamut, and Claire~F Garandeau.
\newblock Bullying prevention in adolescence: {S}olutions and new challenges from the past decade.
\newblock {\em Journal of Research on Adolescence}, 31(4):1023--1046, 2021.

\bibitem{sampson2004seeing}
Robert~J Sampson and Stephen~W Raudenbush.
\newblock Seeing disorder: {N}eighborhood stigma and the social construction of “broken windows”.
\newblock {\em Social Psychology Quarterly}, 67(4):319--342, 2004.

\bibitem{short2008statistical}
Martin~B Short, Maria~R D'orsogna, Virginia~B Pasour, George~E Tita, Paul~J Brantingham, Andrea~L Bertozzi, and Lincoln~B Chayes.
\newblock A statistical model of criminal behavior.
\newblock {\em Mathematical Models and Methods in Applied Sciences}, 18(supp01):1249--1267, 2008.

\bibitem{short2009measuring}
Martin~B Short, Maria~R D’orsogna, Paul~J Brantingham, and George~E Tita.
\newblock Measuring and modeling repeat and near-repeat burglary effects.
\newblock {\em Journal of Quantitative Criminology}, 25:325--339, 2009.

\bibitem{sommer2014bullying}
Friederike Sommer, Vincenz Leuschner, and Herbert Scheithauer.
\newblock Bullying, romantic rejection, and conflicts with teachers: {T}he crucial role of social dynamics in the development of school shootings--a systematic review.
\newblock {\em International Journal of Developmental Science}, 8(1-2):3--24, 2014.

\bibitem{tandoc2019facts}
Edson~C Tandoc~Jr.
\newblock The facts of fake news: {A} research review.
\newblock {\em Sociology Compass}, 13(9):e12724, 2019.

\bibitem{tao2015boundedness}
Youshan Tao and Michael Winkler.
\newblock Boundedness vs. blow-up in a two-species chemotaxis system with two chemicals.
\newblock {\em Discrete and Continuous Dynamical Systems - Series B}, 20(9):3165--3183, 2015.

\bibitem{MATLAB}
{The MathWorks Inc.}
\newblock {MATLAB} version: 24.2.0.2740171 (r2024b) update 1, 2024.

\bibitem{PdeToolbox}
{The MathWorks Inc.}
\newblock Partial differential equation toolbox version: 24.2 (r2024b), 2024.

\bibitem{twemlow2006teachers}
Stuart~W Twemlow, Peter Fonagy, Frank~C Sacco, and John~R Brethour~Jr.
\newblock Teachers who bully students: {A} hidden trauma.
\newblock {\em International Journal of Social Psychiatry}, 52(3):187--198, 2006.

\bibitem{van2014bullying}
Mitch van Geel, Paul Vedder, and Jenny Tanilon.
\newblock Bullying and weapon carrying: {A} meta-analysis.
\newblock {\em JAMA Pediatrics}, 168(8):714--720, 2014.

\bibitem{van2013emotions}
Dunya Van~Troost, Jacquelien Van~Stekelenburg, and Bert Klandermans.
\newblock Emotions of protest.
\newblock In Nicolas Demertzis, editor, {\em Emotions in politics: The affect dimension in political tension}, pages 186--203. Springer, 2013.

\bibitem{wang2015stationary}
Ke~Wang, Qi~Wang, and Feng Yu.
\newblock Stationary and time-periodic patterns of two-predator and one-prey systems with prey-taxis.
\newblock preprint can be found at \url{https://arxiv.org/abs/1508.03909}, 2015.

\bibitem{enwiki:1299623116}
{Wikipedia contributors}.
\newblock Routine activity theory --- {Wikipedia}{,} the free encyclopedia.
\newblock \url{https://en.wikipedia.org/w/index.php?title=Routine_activity_theory&oldid=1299623116}, 2025.
\newblock [Online; accessed 9-July-2025].

\bibitem{WilsonKelling1982}
J.~Wilson and G.~Kelling.
\newblock Broken windows: {T}he police and neighborhood safety.
\newblock {\em Atlantic Monthly}, 249(4):29--38, 1982.

\bibitem{winkler2010absence}
Michael Winkler.
\newblock Absence of collapse in a parabolic chemotaxis system with signal-dependent sensitivity.
\newblock {\em Mathematische Nachrichten}, 283(11):1664--1673, 2010.

\bibitem{winkler2010aggregation}
Michael Winkler.
\newblock Aggregation vs. global diffusive behavior in the higher-dimensional {K}eller--{S}egel model.
\newblock {\em Journal of Differential Equations}, 248(12):2889--2905, 2010.

\bibitem{wouters2019persuasive}
Ruud Wouters.
\newblock The persuasive power of protest. {H}ow protest wins public support.
\newblock {\em Social Forces}, 98(1):403--426, 2019.

\bibitem{wu2016global}
Sainan Wu, Junping Shi, and Boying Wu.
\newblock Global existence of solutions and uniform persistence of a diffusive predator--prey model with prey-taxis.
\newblock {\em Journal of Differential Equations}, 260(7):5847--5874, 2016.

\bibitem{wu2018dynamics}
Sainan Wu, Jinfeng Wang, and Junping Shi.
\newblock Dynamics and pattern formation of a diffusive predator--prey model with predator-taxis.
\newblock {\em Mathematical Models and Methods in Applied Sciences}, 28(11):2275--2312, 2018.

\bibitem{xu2023sight}
Duoduo Xu and Jiao Guo.
\newblock In sight, in mind: {S}patial proximity to protest sites and changes in peoples' political attitudes.
\newblock {\em The British Journal of Sociology}, 74(1):83--104, 2023.

\bibitem{yerlanov2025}
Madi Yerlanov, Qi~Wang, and Nancy Rodr{\'\i}guez.
\newblock Formation and suppression of hotspots in urban crime models with law enforcement.
\newblock {\em Chaos: An Interdisciplinary Journal of Nonlinear Science}, 35(6), 2025.

\bibitem{yerlanov2026}
Madi Yerlanov, Qi~Wang, and Nancy Rodr{\'\i}guez.
\newblock Global bifurcations and pattern formation in target-offender-guardian crime models.
\newblock preprint can be found at \url{https://arxiv.org/abs/2509.18146}, 2025.

\bibitem{zhu2024rumours}
Runping Zhu, Qilin Liu, and Richard Krever.
\newblock Rumours. {W}ho believes them?
\newblock {\em Journal of Information, Communication and Ethics in Society}, 22(2):240--255, 2024.

\end{thebibliography}

\newpage
\appendix
\section{Appendix}
We recap some key results from semigroup theory \cite{horstmann2005boundedness}.
Let $\Lambda$ denote the sectorial operator, \textit{i.e.}, $\Lambda u=-\Delta u$ for $u\in D(\Lambda):=\{f\in W^{2,p}(\Omega)|\partial f/\partial n=0\}$ with $p\in(1,\infty)$ fixed. Then we have the following results from semigroup theory
\begin{lemma}\label{thm:semigroup}

\begin{enumerate}
    \item Let $m\in\{0,1\}$, $q\in[1,\infty]$, and $p\in(1,\infty)$. Then there exists $c_1>0$ such that
    \begin{equation*}
        \norm{u}_{W^{m,q}(\Omega)}\leq c_1 \norm{(\Lambda+1)^\theta u}_{L^p(\Omega)}
    \end{equation*}
    for any $u\in D((\Lambda+1)^\theta)$, and $\theta\in(0,1)$, such that
    \begin{equation*}
        m-\frac{n}{q}<2\theta -\frac{n}{p}.
    \end{equation*}
    \item   Let $p\geq q$. Then there exists $c_2>0$ and $\nu>0$ such that for any $u\in L^q(\Omega)$,
    \begin{equation*}
           \norm{(\Lambda+1)^\theta e^{-t(\Lambda+1)}u}_{L^{p}(\Omega)}\leq c_2t^{-\theta-\frac{n}{2}(\frac{1}{q}-\frac{1}{p})}e^{-\nu t}\norm{u}_{L^q(\Omega)},
    \end{equation*}
    where $\{e^{-t(\Lambda+1)}\}_{t\geq 0}$ maps $L^q(\Omega)$ into $D((\Lambda+1)^\theta)$.
    \item For any $p\in(1,\infty)$ and $\varepsilon>0$, there exists $c_3>0$ and $\mu>0$ such that for any $u\in L^p(\Omega)$,
    \begin{equation*}
           \norm{(\Lambda+1)^\theta e^{-t\Lambda} \nabla \cdot  u }_{L^{p}(\Omega)}\leq c_3t^{-\theta-\frac{1}{2}-\varepsilon}e^{-\mu t}\norm{u}_{L^p(\Omega)}.
    \end{equation*}
\end{enumerate}
\end{lemma}

\begin{remark}
    The above inequalities hold for $D_u \Lambda$ upon rescaling. 
\end{remark}

We state the Sobolev embedding theorem, which is used often throughout this paper.

\begin{lemma}
\textbf{(Gagliardo-Nirenberg in bounded domains)} 
\cite{nirenberg1959elliptic}
Let $\Omega\subset \mathbb{R}^n$ be a measurable, bounded, open, and connected subset satisfying the cone condition. Let $q\in[1,\infty]$, $p\in[1,\infty)$, $r\in[1,\infty]$, $j\in\mathbb{N}\cup\{0\}$, $m\in\mathbb{N}$, $j<m$, $\lambda\in[0,1]$ be such that the following relations
\begin{equation*}
    \frac{1}{p}=\frac{j}{n}+\lambda\Big(\frac{1}{r}-\frac{m}{n}\Big)+\frac{1-\lambda}{q},\quad \frac{j}{m}\leq \lambda\leq 1
\end{equation*}
hold. Then 
\begin{equation*}
    \norm{D^ju}_{L^p(\Omega)}\leq C\big(\norm{D^mu}_{L^r(\Omega)}^\lambda \norm{u}_{L^q(\Omega)}^{1-\lambda}+\norm{u}_{L^\sigma(\Omega)}\big),
\end{equation*}
where $u\in L^q(\Omega)$, $D^mu\in L^r(\Omega)$, and $\sigma$ is arbitrary.
\end{lemma}

We state a special case of the Gagliardo-Nirenberg embedding theorem in bounded domains  for $n=2$
\begin{equation}\label{thm:GNineq}
\norm{u}_{L^2(\Omega)}\leq C\big(\norm{Du}_{L^2(\Omega)}^{1/2} \norm{u}_{L^1(\Omega)}^{1/2}+\norm{u}_{L^{1}(\Omega)}\big). 
\end{equation}    

\end{document}